\documentclass[10pt]{article}
\usepackage{amsmath}
\usepackage{amssymb}
\usepackage{amsthm}
\usepackage[mathcal]{euscript}
\usepackage[cmtip,all]{xy}
\newtheorem{theorem}{\bf Theorem}
\newtheorem{lemma}{\bf Lemma}
\newtheorem{proposition}{\bf Proposition}
\newtheorem{corollary}{\bf Corollary}
\newtheorem{definition}{\bf Definition}
\newtheorem{example}{\bf Example}
\newtheorem{remark}{\bf Remark}
\DeclareMathOperator{\id}{id}
\DeclareMathOperator{\Th}{Th}
\DeclareMathOperator{\Hom}{Hom}

\DeclareMathOperator{\coequalizer}{coequalizer}
\DeclareMathOperator{\holim}{holim}
\DeclareMathOperator{\Aut}{Aut}
\DeclareMathOperator{\Gal}{Gal}

\newcommand{\colim}{\operatorname{colim}}

\begin{document}
\title{On K(1)-local SU-bordism}
\author{by Holger Reeker}
\maketitle

\begin{abstract}
This paper is on an $E_{\infty}$ splitting of the bordism spectrum $M{\rm SU}$. Going the chromatic way, we make use of the resolution of the $K(1)$-local sphere in terms of $KO$-theory. The construction of an ${\rm SU}$-Artin-Schreier class results in an $E_{\infty}$ summand $T_{\zeta}$ in ${\rm SU}$-bordism. Working towards other free summands $TS^0$, we apply the theory of Bott's cannibalistic classes to compute some spherical classes. At the end an alternative calculation of Adams operations in terms of Mahler series in $p$-adic analysis is given. This paper is a shortened version of the author's thesis.
\end{abstract}
\setcounter{tocdepth}{2}
\tableofcontents

\section{Introduction and statement of results}
One of the highlights in algebraic topology was the invention of generalized homology and cohomology theories by Whitehead and Brown in the 1960s. Prominent examples are real and complex $K$-theories first given by Atiyah and Hirzebruch and bordism theories with respect to different structure groups first given by Thom. By Brown's representability theorem every generalized cohomology theory can be represented by a spectrum and these spectra are the center of interest in modern algebraic topology.\\[3mm]
Bordism theories with respect to some structure group $G$, e.g. $G=O$, $SO$, $U$, $SU$, $Sp$, and $Spin$ are defined as follows: Let $M$ be a smooth, closed, $n$-dimensional manifold and $G=\{G_n\}$ be a sequence of topological groups with maps $G_n\rightarrow G_{n+1}$ compatible with their orthogonal representations $G_n\rightarrow O(n)$.
\begin{definition} A $G$-structure on $M$ is a homotopy class of lifts $\tilde{\nu}$ of the classifying map of the stable normal bundle $\nu$
\[
\xymatrix@R-.3cm@C-.2cm@M+.1cm
{  &  BG \ar[d] \\
M \ar[r]^{\nu}\ar@{.>}[ur]^{\tilde{\nu}} & BO 
}\]
A manifold $M$ together with a $G$-structure is called a $G$-manifold.
\end{definition}
For each of the classical groups this gives us the $G$-bordism ring $\Omega_*^G$ and a Thom spectrum $MG$ with $\Omega_*^G=MG_*=MG_*(pt)=\pi_*MG$. Further we have a homology theory $MG_*(-)$ and a cohomology theory $MG^*(-)$. Since we have inclusion maps on group level and since the Thom construction is functorial we get the following tower:
\[
\xymatrix@R-.3cm@C-.2cm@M+.1cm
{MSU\ar[r]\ar[d]  &  MSpin \ar[d] \\
MU\ar[r] & MSO\ar[d]\\
   & MO
}\]
On the level of homotopy one knows at least rationally that the coefficient groups are polynomial rings and one asks for a decomposition on the level of spectra. In 1966 Andersen, Brown and Peterson gave an additive 2-local splitting of $MSpin$
\begin{equation*}
MSpin_{(2)} \simeq \bigvee_{n(J)\text{ even, }1\notin J} ko\langle 4n(J)\rangle  \vee \bigvee_{n(J)\text{ odd, }1\notin J} ko\langle 4n(J)+2\rangle  \vee \bigvee_{i\in I} \Sigma^{d_i} H\mathbb{Z}/2
\end{equation*}
with $J=(i_1,...,i_k)$ a finite sequence and $n(J)=i_1+...+i_k$.
Bordism theories are multiplicative homology theories and their Thom spectra are ring spectra. Moreover they admit even richer structures called $E_{\infty}$ structures, i.e.\ not only the coherent diagrams of commutativity and associativity commute up to homotopy but there are also diagrams of higher coherence. These $E_{\infty}$ structures should be taken into account and therefore we are interested in a splitting in the category of $E_{\infty}$ ring spectra.\\[3mm]
Unfortunately this access raises several other difficulties. Analysing the above additive splitting of 2-local spin bordism by Anderson, Brown and Peterson, the Eilenberg-MacLane part $H\mathbb{Z}/2$ turns out to be a difficult problem. In this situation the modern viewpoint is to apply chromatic homotopy theory and to look at the chromatic tower or at certain monochromatic layers. In our case we consider localizations with respect to the first Morava $K$-theory $K(1)$. At $p=2$ we have
$$ L_{K(1)} \cong L_{S\mathbb{Z}/2} L_{K_{(2)}} $$
and the Eilenberg-MacLane part disappears. This is our approximation to bordism theories. Algebraically this access offers a lot of extra structure since $\pi_0 E$ of a $K(1)$-local $E_{\infty}$ ring spectrum $E$ admits a $\theta$-algebra structure.\\[3mm]
In \cite{Habi} Laures gives a $K(1)$-local splitting of $E_{\infty}$ spectra
$$MSpin\cong T_{\zeta} \wedge \bigwedge_{i=1}^{\infty} TS^0$$
where $T$ is the free functor left adjoint to the forgetful functor from $E_{\infty}$ spectra to spectra and $\wedge$ is the coproduct in the category of $E_{\infty}$ spectra with $\bigwedge TS^0\cong T(\bigvee S^0)$. Such a splitting is also desireable for other bordism theories and a lot of different techniques are involved to get such a splitting.\\[3mm]
In this work we study $K(1)$-local $SU$ bordism. A main result is detecting an $E_{\infty}$ summand $T_{\zeta}$ for a nontrivial element $\zeta\in\pi_{-1} L_{K(1)} S^0\cong \mathbb{Z}_2$
\[
\xymatrix@R-.3cm@C-.2cm@M+.1cm
{ TS^{-1} \ar[d]^{T*}\ar[r]^{\zeta} & S^0 \ar[d]\ar@/^/[ddr] &\\
TD^0=S^0 \ar[r]\ar@/_/[drr] & T_{\zeta} \ar@{.>}[dr] & \\
 & & MSU
}\]
meaning that $T_{\zeta}$ is the resulting $E_{\infty}$ spectrum when attaching a $0$-cell along $\zeta$. To this end, we construct an Artin-Schreier class $b\in KO_0MSU$ satisfying $\psi^3 b=b+1$ which implies that $\zeta=0$ in $\pi_{-1} MSU$.\\[3mm]
Another important result is the construction of spherical classes in $K_*MSU$. Although we do not have a complete splitting, comparison with the spin bordism case shows that spherical classes play an important role: They correspond to free $E_{\infty}$ summands $TS^0$. In this work, we perform the construction of spherical classes via calculations of Adams operations on $K_*(\mathbb{CP}^{\infty}\times \mathbb{CP}^{\infty})$ whose module generators map to the algebra generators of $K_* BSU$. Later we can use Bott's theory of cannibalistic classes to lift the Adams operations to the level of Thom spectra.\\[3mm]
Since the $K$-homology of $\mathbb{CP}^{\infty}$ is isomorphic to the ring of numerical polynomials, we are able to provide an alternative calculation of the Adams operations on $K_*\mathbb{CP}^{\infty}$ using Mahler series expansion in $p$-adic analysis.

\subsubsection*{Acknowledgements}
First of all, I would like to thank my supervisor Prof.~Dr.\ Gerd Laures for introducing me to the field of $K(1)$-local $E_{\infty}$ spectra with all their inter\-esting arithmetic. I would like to express my profound respect to Prof.~Dr.\ Uwe Abresch for his spontaneous willingness to act as co-referee. At the same time, I want to say thank you to all members of the chair of topology at the Ruhr-Universit\"at Bochum for the good atmosphere and for numerous mathematical and non-mathematical discussions -- my special thanks go to Dr.~Markus Szymik, Dr.~Hanno von Bodecker, Jan M\"ollers, Norman Schumann and Sieglinde Fernholz. I appreciate the financial support from the DFG within the Graduiertenkolleg 1150 ``Homotopy and Cohomology".

\section{Some homotopical algebra}
\subsection{Generalized cohomology theories and spectra}
In this section we want to recall the basic notations of generalized cohomology theories and spectra as their representing objects. We will see the correspondence between them and have a look at their fundamental properties. The relevant homotopy category is the stable homotopy category.
\begin{definition}
A generalized cohomology theory $E$ consists of a sequence $\{E^n\}_{n\in\mathbb{Z}}$ of contravariant homotopy functors
$$ E^n: \mathbf{CWPairs}\rightarrow \mathbf{AbGroups} $$
together with natural transformations
$$ \delta: E^n(X)\rightarrow E^{n+1}(X,A)$$
satisfying the axioms
\begin{itemize}
\item {\bf Excision:} The projection $(X,A)\rightarrow X/A$ induces an isomorphism $$\tilde{E}^n(X/A)\rightarrow E^n(X,A)$$ for all pairs $(X,A)$.
\item {\bf Exactness:} The long sequence of abelian groups
$$ ...\rightarrow E^n(X,A)\rightarrow E^n(X)\rightarrow E^n(A) \stackrel{\delta}{\rightarrow} E^{n+1}(X,A)\rightarrow ...$$
is exact for all pairs $(X,A)$.
\item {\bf Strong additivity:} For every family of spaces $\{X_i\}_{i\in I}$ the natural map
$$ E^n(\coprod_{i\in I} X_i) \rightarrow \prod_{i\in I} E^n(X_i) $$
is an isomorphism.
\end{itemize}
\end{definition}

\begin{proposition}
Every generalized cohomology theory $E$ enjoys the following properties:
\begin{enumerate}
\item For a pointed topological space $X$ there is a natural direct sum splitting $$E^n(X)\cong \tilde{E}^n(X)\oplus E^n(*).$$
\item For a family $\{X_i\}_{i\in I}$ of pointed topological spaces the map of reduced cohomology groups
$$ \tilde{E}^n(\bigvee_{i\in I} X_i)\rightarrow \prod_{i\in I} \tilde{E}^n(X_i)$$
is an isomorphism.
\item For a pointed topological space $X$ we have natural isomorphisms
\[\xymatrix@R-.3cm@C-.2cm@M+.1cm{ 
\tilde{E}^n(X) \ar[rr]^{\cong}_{\delta} & & E^{n+1}(CX,X) & & \ar[ll]^{\text{excision}}_{\cong}  \tilde{E}^{n+1}(\Sigma X) }\]
\item {\bf Mayer-Vietoris:} For $X=X_1\cup X_2$ (open covering) we have the long exact sequence
$$ ...\rightarrow E^n(X) \stackrel{(i_1^*,i_2^*)}{\rightarrow} E^n(X_1)\oplus E^n(X_2) \stackrel{j_1^*-j_2^*}{\rightarrow} E^n(X_1\cap X_2) \stackrel{\delta}{\rightarrow} E^{n+1}(X)\rightarrow ...$$
\item {\bf Milnor sequence:} For a filtration $X=\colim X_i$ we get a short exact sequence with the derived limit
$$ 0\rightarrow \lim\,^{1} E^{n-1}(X_i)\rightarrow E^n(X) \rightarrow \lim E^n(X_i)\rightarrow 0 $$
which detects phantom maps. 
\end{enumerate}
\end{proposition}
These cohomology functors are representable by a sequence of spaces and with the suspension isomorphism we naturally get the following definition:
\begin{definition} A spectrum $X$ is a sequence of pointed topological spaces $X_0,X_1,X_2,...$ together with structure maps
$$ \sigma_n: X_n\wedge S^1\rightarrow X_{n+1}$$
or the adjoint map $\tilde{\sigma}: X_n\rightarrow \Omega X_{n+1}$ respectively. If $\tilde{\sigma}$ is a weak equivalence $X$ is called an $\Omega$-spectrum.
\end{definition}
By Brown's representability theorem every generalized cohomology theory can be represented by an $\Omega$-spectrum. On the other hand every spectrum defines a cohomology (and homology) theory. It is worth mentioning that every spectrum can functorially be turned into an $\Omega$-spectrum. As an illustrative example, we define an $\Omega$-spectrum for complex K-theory.

\begin{example} ($K$-theory) First of all we make use of Bott periodicity, i.e.\ there is a homotopy equivalence $\Omega^2 BU\cong \mathbb{Z}\times BU$, and we define an $\Omega$-spectrum $K$ by setting
$$ K_n=\begin{cases} \mathbb{Z}\times BU & \text{if n is even,}\\ \Omega BU & \text{if n is odd} \end{cases}$$
with structure maps adjoint to 
$$ (\tilde{\sigma}: K_n\rightarrow \Omega K_{n+1})=\begin{cases} \text{the Bott equivalence } \quad \mathbb{Z}\times BU \stackrel{\cong}{\longrightarrow} \Omega^2BU & \text{if n is even},\\ \text{the identification }\, \qquad \Omega BU \stackrel{\cong}{\longrightarrow} \Omega(\mathbb{Z}\times BU) & \text{ if n is odd}.\end{cases}$$
This is called the complex topological $K$-theory spectrum. Its homotopy groups are
$$ \pi_n K = \begin{cases} \pi_0(\mathbb{Z}\times BU)\cong \mathbb{Z} & \text{if n is even}\\ \pi_1 BU=0 & \text{if n is odd.}\end{cases}$$
\end{example}
\begin{example}($KO$-theory) Similarly, we obtain the real topological $K$-theory spectrum $KO$ using real Bott periodicity, i.e.
$$ \Omega^8 BO\cong \mathbb{Z}\times BO.$$
Its homotopy groups are given in the following table:
\begin{center}
\begin{tabular}{c||l|l|l|l|l|l|l|l}
n mod 8 & 0     & 1 & 2 & 3 & 4 & 5 & 6 & 7 \\
\hline
$\pi_n KO$ & $\mathbb{Z}$ & $\mathbb{Z}/2$ & $\mathbb{Z}/2$ & 0 & $\mathbb{Z}$ & 0 & 0 & 0
\end{tabular}
\end{center}
\end{example}
In the above examples we recalled the additive homotopy groups, but as we know there is also a multiplicative structure. For example the coefficient rings of $K$ and $KO$ are $\pi_*K=\mathbb{Z}[u^{\pm 1}]$ with the Bott element $u\in\pi_2K$ as invertible element and
$$ \pi_* KO=\mathbb{Z}[\eta,\alpha,\beta^{\pm 1}]/(2\eta=0,\eta^3=0,\eta\alpha=0,\alpha^2=4\beta).$$
\begin{definition}
A cohomology theory $E^*$ is called multiplicative if it is equipped with a product
$$ \times : \tilde{E}^p(X)\otimes \tilde{E}^q(Y) \rightarrow \tilde{E}^{p+q}(X\wedge Y)$$
which is associative, graded commutative, unital and stable.
\end{definition}
A multiplicative theory is realized by a ring spectrum, i.e.\ a spectrum $E=(E_n)_n$ together with maps
$$ \mu_{mn}:E_m\wedge E_n\rightarrow E_{m+n} \quad\text{ and } \eta_n: S^n\rightarrow E_n$$
such that the following diagrams representing the properties {\em associativity, commutativity, unit} and {\em stability} commute up to homotopy:
\[
\xymatrix@R-.3cm@C-.2cm@M+.1cm
{E_k\wedge E_m\wedge E_n \ar[r]\ar[d]  &  E_{k+m}\wedge E_n \ar[d] & &   E_m\wedge E_n\ar[rr]\ar[dr] && E_n\wedge E_m\ar[dl]  \\
E_k\wedge E_{m+n}\ar[r] & E_{k+m+n}                                & &  &              E_{m+n} & \\
S^n\wedge E_n \ar[dr]\ar[r]  &  E_m\wedge E_n \ar[d]^{\mu} & E_m\wedge S^n \ar[l]\ar[dl]&  &  \Sigma S^n\ar[r]\ar[d] & \Sigma E_n\ar[d]  \\
                              &  E_{m+n}                    &                            &  &  S^{n+1}\ar[r]          & E_{n+1}   
}\]
With these notations the product
$$ \times : \tilde{E}^p(X)\otimes \tilde{E}^q(Y) \rightarrow \tilde{E}^{p+q}(X\wedge Y)$$
of the multiplicative cohomology theory $E^*$ is given by
$$ (f:X\rightarrow E_m\,,\, g:Y\rightarrow E_n)\mapsto (f\wedge g:X\wedge Y\rightarrow E_m\wedge E_n\stackrel{\mu_{mn}}{\longrightarrow} E_{m+n}).$$
Having the above notations for associativity and commutativity of ring spectra in mind, one might think of higher coherence conditions (i.e.\ having a smash product of four or more spaces we want to have commutativity up to homotopy when evaluating in different order). This leads to the notion of an $E_{\infty}$ ring spectrum, which comes with an $E_{\infty}$ operad controlling the coherence. This is the sort of extra structure all our spectra (bordism spectra, $K$-theory spectra, Eilenberg-MacLane spectra) have and in this world our bordism splitting takes place. Another (in fact Quillen equivalent) model for $E_{\infty}$ spectra are symmetric spectra which come up in the next paragraph.

\subsection{Symmetric spectra over topological spaces}
There are a lot of highly structured ring spectra: $E_{\infty}$ spectra, $\mathbb{S}$-algebras, symmetric ring spectra and strict commutative ring spectra. They are all Quillen equivalent and their homotopy category is the stable homotopy category. Therefore it does not matter which model we use, but it gives a good feeling to have safe foundations. In this section we consider sequential spectra over pointed topological spaces and refer to \cite{HSS00}, \cite{EKMM} and \cite{Schwede08}.
Let $\mathcal{T}_*$ denote the category of pointed topological spaces.
\begin{definition} A symmetric spectrum $X$ consists of
\begin{itemize}
\item a sequence $X_0,X_1,...\in\mathcal{T}_*$
\item structure maps $\sigma: X_n\wedge S^1\rightarrow X_{n+1}$
\item symmetric operations $\Sigma_n \curvearrowright X_n$ such that
\[
\xymatrix@R-.3cm@C-.2cm@M+.1cm 
{\sigma^p: X_n\wedge S^p \ar[rr]^{\sigma\wedge S^{p-1}} && X_{n+1}\wedge S^{p-1} \ar[rr]^(0.65){\sigma \wedge S^{p-2}}&& ...\ar[r]^{\sigma} & X_{n+p}}
\]
are $\Sigma_n \times \Sigma_p$-equivariant.
\end{itemize}
A map $f:X\rightarrow Y$ of symmetric spectra is a family of maps $f_n:X_n\rightarrow Y_n$ of $\Sigma_n$-equivariant maps such that
\[
\xymatrix
@R-.3cm@C-.2cm@M+.1cm 
{X_n\wedge S^1 \ar[d]_{f_n\wedge S^1}\ar[r]^{\sigma} & X_{n+1}\ar[d]^{f_{n+1}} \\
Y_n\wedge S^1 \ar[r]^{\sigma} & Y_{n+1}}
\]
commutes. This gives us the category of symmetric spectra $Sp^{\Sigma}$.
\end{definition}

\begin{example}[Suspension spectrum]
For a pointed topological space $X\in\mathcal{T}_*$ the suspension spectrum $\Sigma^{\infty}X$ is defined by $(\Sigma^{\infty}X)_n:=X\wedge S^n$ and the structure map is given by the identity morphism $$X\wedge S^n\wedge S^1\rightarrow X\wedge S^{n+1}. $$
\end{example}
\begin{example}[Sphere spectrum] The sphere spectrum $\mathbb{S}=(S^0,S^1,S^2,...)$ is the suspension spectrum for $K=S^0$.
\end{example}
\begin{example}[Eilenberg-MacLane spectrum $H\mathbb{Z}$] With $S^1=\Delta^1/\partial\Delta^1$ the n-sphere is given a simplicial structure by $S^n=S^1\wedge ... \wedge S^1$. Then let $(\mathbb{Z}\otimes S^n)_k$ be the free abelian group on the unpointed k-simplices of $S^n$. Define the Eilenberg-MacLane spectrum $H\mathbb{Z}$ by $(H\mathbb{Z})_n:=|\mathbb{Z}\otimes S^n|$ to be the realization of the simplicial abelian group. $H\mathbb{Z}_n$ is a $K(\mathbb{Z},n)$ since for every simplicial abelian group $\pi_n|A|=H_nA$ and here we have $$\pi_m|\mathbb{Z}\otimes S^n| =H_m \mathbb{Z}\otimes S^n=H_m C_.S^n=\begin{cases}\mathbb{Z} & \text{for} n=m\\ 0 & \text{for} n\neq m \end{cases}$$ with $C_.S^n$ the singular chain complex. The action $\Sigma_n\curvearrowright S^n=S^1 \wedge...\wedge S^1$ is given by permuting the factors and the structure maps are induced by $S^n\wedge S^1\rightarrow S^{n+1}$.
\end{example}
\begin{example}[Unoriented bordism spectrum $MO$]
The construction is given for bordism theory with respect to the orthogonal group but generalizes to other groups in the obvious way. Construct $EO(n):=|k\rightarrow O(n)^{k+1}|$ as the realization of the simplicial complex. The orthogonal group $O(n)$ acts on this space by multiplication on the right. This gives us the classifying space $$BO(n):=EO(n)/O(n).$$ 
Take the associated bundle $$\xi_n: EO(n)\times_{O(n)} \mathbb{R}^n \rightarrow EO(n)/O(n)=BO(n)$$
and define its Thom space 
$$Thom(\xi_n) = D_{\xi}/S_{\xi}=\frac{EO(n)\times_{O(n)} D^n}{EO(n)\times_{O(n)} S^{n-1}}\cong EO(n)_+ \wedge_{O(n)} (D^n/S^{n-1}). $$
Since $D^n/S^{n-1}$ is $O(n)$-equivariantly isomorphic to $\mathbb{R}\cup\{\infty\}$ and $O(n)$ acts on $\mathbb{R}\cup\{\infty\}$ preserving $\{\infty\}$ we have an $O(n)$-action on $S^n$. Defining the $n^{th}$ space of $MO$ to be $MO_n:=EO(n)_+\wedge_{O(n)} S^n$, the Thom space of $\xi_n$ gives a symmetric spectrum with symmetric operations $\Sigma_n\subset O(n)$ coming from coordinate permutations.
\end{example}

\begin{definition}[Symmetric sequences] A symmetric sequence consists of a sequence 
$$X_0,X_1,...\in\mathcal{T}_*$$ 
and symmetric operations $\Sigma_n\curvearrowright X_n$. A map $f:X\rightarrow Y$ is a family of $\Sigma_n$-equivariant maps $f_n:X_n\rightarrow Y_n$. This category is denoted by $\mathcal{T}_*^{\Sigma}$.
\end{definition}

For $X,Y\in\mathcal{T}_*^{\Sigma}$ we can define their tensor product $X\otimes Y$ by
$$ (X\otimes Y)_n:= \bigvee_{p+q=n} (\Sigma_n)_+ \wedge_{\Sigma_p\times \Sigma_q} X_p\wedge Y_q$$
with $\Sigma_p\times \Sigma_q$-diagonal operation: for $(g,h)\in\Sigma_p\times \Sigma_q\subset \Sigma_n$, $\alpha\in\Sigma_n$
let $(\alpha(g,h),x,y)\sim (\alpha,gx,hy)$ be equivalent. The so defined tensor product admits a unit $U=(S^0,*,*,...)$
$$(U\otimes X)_n=\bigvee_{p+q=n} (\Sigma_n)_+\wedge_{\Sigma_p\times\Sigma_q} U_p\wedge X_q\cong(\Sigma_n)_+\Sigma_n X_n\cong X_n,$$ hence $U\otimes X\cong X$. Furthermore the tensor product admits a twist isomorphism $\tau:X\otimes Y\rightarrow Y\otimes X$ sending $(\alpha,x,y)$ to $(\alpha^s,y,x)$ with 
$$ s(1,...,q,q+1,...,q+p)=(p+1,...,p+q,1,...,p)$$ 
the $(p,q)$-shuffle.

\begin{proposition}
$(\mathcal{T}_*^{\Sigma},\otimes,\tau)$ is a symmetric monoidal category.
\end{proposition}

In the following we want to define a smash-product in the category of symmetric spectra $Sp^{\Sigma}$. Let $S:=(S^0,S^1,...)$ be the symmetric sphere sequence.
\begin{proposition}
$S$ is a commutative monoid in $\mathcal{T}_*^{\Sigma}$, i.e. there exist maps $\mu:S\otimes S\rightarrow S$ and $\eta:U\rightarrow S$ such that
\[
\xymatrix@R-.3cm@C-.2cm@M+.1cm 
{S\otimes S \ar[dr]_{\mu} \ar[rr]^{\tau} && S\otimes S\ar[dl]^{\mu} \\ & S &}
\]
commutes.
\end{proposition}
To get an idea where this multiplication map $\mu$ comes from, recall that 
$$\Hom^*_G(\bigvee X_j, Y)\cong \prod \Hom_G^*(X_j, Y)$$ 
and for a subgroup $H\subset G$ we have 
$$\Hom_G^*(G_+\wedge_H X,Y)\cong \Hom_H^*(X,res_H^G Y).$$
Hence the map $\mu:S\otimes S\rightarrow S$ reduces to maps $$\mu_n:\bigvee_{p+q=n} (\Sigma_n)_+\wedge_{\Sigma_p\times\Sigma_q} S^p\wedge S^q\rightarrow S^n$$ which restrict to $\tilde{\mu}_n:S^p\wedge S^q\rightarrow S^n$ when considering the Young-subgroups $\Sigma_p\times\Sigma_q\subset\Sigma_n$.

\begin{definition}[Category of left $S$-modules]
A left $S$-module is a symmetric sequence $X\in\mathcal{T}_*^{\Sigma}$ with a map $S\otimes X\rightarrow X$ such that
\[
\xymatrix@R-.3cm@C-.2cm@M+.1cm 
{(S\otimes S)\otimes X \ar[dr]_{\mu\otimes id} \ar[rr] && S\otimes (S\otimes X)\ar[rr]^{id\otimes m} && S\otimes X\ar[dl]^m \\ & S\otimes X \ar[rr]^m && X &}
\]
commutes.
\end{definition}

This gives us an equivalence of categories:
\begin{center}
		\begin{tabular}{lcl}
			left S-modules $X\in\mathcal{T}_*^{\Sigma}$& $\longleftrightarrow$ & symmeric spectra $X\in Sp^{\Sigma}$\\
			$(S\otimes X)_n \stackrel{m_n}{\rightarrow} X_n$ & & $S^p\wedge X_q\rightarrow X_{p+q}$ \\
			 & &  $\Sigma_p\times\Sigma_q$-equivariant
		\end{tabular}
\end{center}
		
\begin{definition}[Smash-Product] Let $X,Y\in Sp^{\Sigma}\cong\text{\bf left S-mod}$. Then
\[\xymatrix@R-.3cm@C-.2cm@M+.1cm 
{X\wedge Y:= X\otimes_S Y:=\coequalizer(X\otimes S\otimes Y\ar@<.5ex>[r]^(0.8){1_X\otimes m_Y}\ar@<-.5ex>[r]_(0.8){(m_X\circ\tau)\otimes 1_Y} &X\otimes Y)} \]
is a left $S$-module since $(S\otimes-)$ preserves colimits.
\end{definition}

\begin{proposition}
$(Sp^{\Sigma},\wedge)$ is a symmetric monoidal category.
\end{proposition}

\subsection{Complex oriented theories and computational methods}
Complex oriented theories are special generalized cohomology theories which have a big advantage: They are computable. We briefly recall some basic results. Let $E$ be a multiplicative cohomology theory.
\begin{definition}
$E$ is called complex orientable if there is a class $x\in \tilde{E}^2\mathbb{CP}^{\infty}$ mapping to $1\in \tilde{E}^0S^0$ under the map
\[
\xymatrix@R-.3cm@C-.2cm@M+.1cm{
\tilde{E}^2\mathbb{CP}^{\infty}\ar[r] & \tilde{E}^2 \mathbb{CP}^1\cong \tilde{E}^2 S^2\ar[r]^(.65){\Sigma^{-2}} & \tilde{E}^0S^0 
 }\] 
induced by the inclusion $\mathbb{CP}^1 \hookrightarrow \mathbb{CP}^{\infty}$. Any choice of $x$ is a complex orientation of $E$.
\end{definition}
\begin{proposition}
The map $E^*[x]/(x^{n+1})\rightarrow E^*(\mathbb{CP}^n)$ mapping $x$ to $x_{|\mathbb{CP}^n}$ is an isomorphism. 
\end{proposition}
\begin{proof}
The above map is well defined since one can cover $\mathbb{CP}^n$ with open contractible sets $U_0,...,U_n$ implying the existence of $x_i\in E^2(\mathbb{CP}^n,U_i)$ with $x_{i|\mathbb{CP}^n}=x_{|\mathbb{CP}^n}$. Multiplying all these
$$ x_0 \cdots x_n \in E^{2(n+1)}(\mathbb{CP}^n, \bigcup_{i=0}^n U_i)=0$$
shows that $x^{n+1}_{|\mathbb{CP}^n}=(x_0\cdots x_n)_{|\mathbb{CP}^n}=0$. Next we set up the Atiyah-Hirzebruch spectral sequence
$$ E_2^{p,q} = H^p(\mathbb{CP}^n, E^q(pt)) \Rightarrow E^{p+q} (\mathbb{CP}^n); $$
with $F^{p,q}=\ker(E^{p+q}X\rightarrow E^{p+q}X_{p-1})$ and $X=\mathbb{CP}^n$ we consider
$$ F^{2,0}=\ker(E^2\mathbb{CP}^n\rightarrow E^2(pt))\ni x_{|\mathbb{CP}^n} $$
and look at the associated graded $F^{2,0}/F^{3,-1}\cong E^{2,0}_{\infty}\subset E_2^{2,0}$. We have the map
$$ E_{\infty}^{2,0}=H^2(\mathbb{CP}^1,E^0(pt))=E^0(pt)\cdot t$$ given by $x_{|\mathbb{CP}^n}\mapsto t$. By multiplicativity the spectral sequence collapses and we have the following isomorphism of graded rings
$$ E^*(\mathbb{CP}^n)\cong E^*[x_{|\mathbb{CP}^n}]/(x^{n+1}_{|\mathbb{CP}^n}).$$
\end{proof}
\begin{corollary}
$E^*\mathbb{CP}^{\infty}\cong E^*[\![x]\!]$.
\end{corollary}
\begin{corollary}
$E^*(\mathbb{CP}^{\infty}\times ... \times \mathbb{CP}^{\infty})\cong E^*[\![x_1,...,x_n]\!]$.
\end{corollary}
\begin{theorem}
Let $E$ be a complex oriented cohomology theory, $X$ a space and $\xi\rightarrow X$ a complex vector bundle. Then there exists a unique system of cohomology classes $c_i(\xi)\in\tilde{E}^{2i}(X)$ with the following properties:
\begin{itemize}
\item \textup{(}normalization\textup{)} For the tautological bundle $\lambda^*\rightarrow \mathbb{CP}^{\infty}$ we have $c_1(\lambda^*)=x$.
\item \textup{(}naturality\textup{)} For all maps $f:X\rightarrow Y$ we have $f^*c_i(\xi)=c_i(f^*\xi)$.
\item \textup{(}Cartan formula\textup{)} For the total Chern class $c=1+c_1+c_2+...$ we have $c(\xi\oplus\eta)=c(\xi)c(\eta)$.
\end{itemize}
\end{theorem}

\begin{lemma} Let $E$ be a complex oriented theory with a complex orientation $x\in E^2\mathbb{CP}^{\infty}$:
\begin{enumerate}
\item $E^*\mathbb{CP}^{\infty} \cong E^*[\![x]\!]$ the power series ring in $x$ over $E^*$.
\item $E^*(\mathbb{CP}^{\infty}\times \mathbb{CP}^{\infty})=E^*\mathbb{CP}^{\infty} \hat{\otimes}_{E^*} E^*\mathbb{CP}^{\infty}$.
\item $E_*\mathbb{CP}^{\infty}$ is a free $E^*$ module with generators $\beta_i\in E_{2i}\mathbb{CP}^{\infty}$, $i\ge 0$ dual to $x^i$, i.e. $\langle x^i,\beta_j\rangle = \delta_{ij}$.
\item $E_*(\mathbb{CP}^{\infty}\times \mathbb{CP}^{\infty})=E_*\mathbb{CP}^{\infty} \otimes_{E_*} E_*\mathbb{CP}^{\infty}$.
\item The diagonal $\mathbb{CP}^{\infty}\rightarrow \mathbb{CP}^{\infty}\times\mathbb{CP}^{\infty}$ induces a coproduct $\psi$ on $E_*\mathbb{CP}^{\infty}$ with $\psi(\beta_n)=\sum_{i+j=n} \beta_i\otimes \beta_j$.
\end{enumerate}
\end{lemma} 

\subsection{Bousfield localization of spectra}
Bousfield localization theory is an analogue and a generalization of arithmetic localization theory in algebra. While arithmetic localization takes place for example in the category of rings and is done with respect to some prime ideals, Bousfield localization theory takes place in the stable homotopy category $\mathit{SHC}$ and can be done with respect to any generalized homology theory $E_*$ (represented by the spectrum $E$).
\begin{definition}
A map of spectra $f: X\rightarrow Y$ is called an $E$-equivalence if it induces an isomorphism $f_*:E_*X\stackrel{\cong}{\rightarrow} E_*Y$ in $E_*$-homology. With this a spectrum $Z$ is called $E$-local if it has the $E$-extension property for every $E$-equivalence $f:X\rightarrow Y$:
\[\xymatrix@R-.3cm@C-.2cm@M+.1cm{ 
X \ar[r]\ar[d]_f & Z \\
Y \ar@{.>}[ur]_{\exists !}  &
}\]
Equivalently $Z$ is $E$-local if the functor $[- ,Z ]$ takes $E$-equivalences to isomorphisms. 
\end{definition}
\begin{definition}
We call $\gamma_E(X): X\rightarrow X_E$ an $E$-localization if
\begin{enumerate}
\item $\gamma_E$ is an $E$-equivalence
\item $X_E$ is $E$-local.
\end{enumerate}
\end{definition}
\begin{remark}
This definition is equivalent to the definition given by Bousfield \cite{Bousfield79} and Ravenel \cite{Rav84}\textup{:} A spectrum $Y$ is $E$-local if for each $E$-acyclic spectrum $X$ \textup{(}i.e.\ $E_*X=0$\textup{)} follows\textup{:} $[X,Y]=0$. The reason is that a map $f:X\rightarrow Y$ gives the cofiber sequence
$$ X \stackrel{f}{\rightarrow} Y \rightarrow Z=cofiber(f) $$
and a long exact sequence
$$ ... \rightarrow [Z,T]\rightarrow [Y,T] \rightarrow [X,T] \rightarrow ...$$
i.e.\ having $[X,T]=0$ for an $E$-acyclic $X$ we get isomorphisms $[Z,T]\stackrel{\cong}{\rightarrow} [Y,T]$ and thus a lift $Z\rightarrow T$ \textup{(}and vice versa\textup{)}.
\end{remark}
It follows directly from the definition that $\gamma_E(X)$ (if it exists) is unique up to isomorphism. The existence is given via
\begin{theorem}[Bousfield] With the above notation we have
\begin{enumerate}
\item $\gamma_E(X)$ always exists
\item $\gamma_E(X)$ assembles an idempotent functor
$$ L_E: \mathit{SHC}\rightarrow \mathit{SHC}_{E\text{-local objects}}$$
\item $L_E$ is (up to equivalence) the categorical localization of $\mathit{SHC}$ with respect to $E$-equivalences, i.e. given a functor $\mathit{SHC}\rightarrow \mathit{D}$ such that $E$-equivalences are inverted we get the commutative diagram
\[\xymatrix@R-.3cm@C-.2cm@M+.1cm{ 
\mathit{SHC} \ar[rr]^{E\text{-equiv. inv.}} \ar[d] & & \mathit{D} \\
L_E \mathit{SHC} \ar@{.>}[urr]_{\exists !}  & &
}\]
\end{enumerate}
\end{theorem}

Collecting results from \cite{Bousfield79}, \cite{EKMM} and \cite{Habi} we formulate an overview theorem:
\begin{theorem} The localization functor $L_E$ has the following properties:
\begin{enumerate}
\item it is idempotent, i.e. $L_E L_E=L_E$
\item if $W\rightarrow X\rightarrow Y$ is a cofiber sequence, so is $L_E W\rightarrow L_E X\rightarrow L_E Y$
\item the homotopy inverse limit of $E$-local spectra is $E$-local
\item if $E$ is a ring spectrum and $X$ is a $E$-module spectrum then $X$ is $E$-local
\item the localization functor $L_E$ can be chosen to preserve $E_{\infty}$-structures
\item at $p=2$ we have $ L_{K(1)} = L_{S\mathbb{Z}/2} L_{K_{(2)}}$
\end{enumerate}
\end{theorem}

Next we introduce the so-called Bousfield classes which compare the localization functors.
\begin{definition}[Bousfield classes]
For spectra $E$ and $F$ we say $E\ge F$ if $f:X \stackrel{\sim_E}{\rightarrow} Y$ implies $f: X \stackrel{\sim_F}{\rightarrow} Y$. This defines an equivalence relation by
$$ E\sim F \quad \text{if } E\ge F \text{ and } F\ge E.$$
The equivalence class $\langle E\rangle $ is called the Bousfield class of $E$.
\end{definition}
Note that the relation $E\ge F$ gives a canonical factorization:
\[\xymatrix@R-.3cm@C-.2cm@M+.1cm{ 
                              & & L_E \mathit{SHC}  \ar@{.>}[d]^{\exists !} \\
 \mathit{SHC} \ar[rr]\ar[urr] & & L_F \mathit{SHC}    
}\]

\begin{example}[$p$-localization]
For the Moore spectrum $E=M\mathbb{Z}_{(p)}$ \textup{(}i.e.\ $E$ is connected $\pi_{<0}E=0$ and the only non-vanishing homology is in degree zero $H_0(E)\cong \mathbb{Z}_{(p)}$\textup{)} we have
$$ (M\mathbb{Z}_{(p)})_* X\cong \pi_*X \otimes \mathbb{Z}_{(p)},$$
and $M\mathbb{Z}_{(p)}$-localization is realized by
$$ X\mapsto X\wedge M\mathbb{Z}_{(p)}, $$
which we also call $p$-localization due to its effect in homotopy. 
\end{example}

\begin{example}[rationalization]
Similarly we have 
$$M\mathbb{Q}_*(X)=\pi_*(X)\otimes \mathbb{Q}$$
and rationalization is given by
$$ X\mapsto X\wedge M\mathbb{Q}.$$
In particular $L_{\mathbb{Q}} S=M\mathbb{Q}$.
\end{example}

\begin{definition}[smashing spectrum]
If $L_E X=X\wedge L_E S$, the spectrum $E$ is called smashing.
\end{definition}
We have seen that the above Moore spectra are smahing and will for reasons of notation denote the rationalization of a spectrum $X$ by $X_{\mathbb{Q}}$ and the $p$-localization by $X_{(p)}$. With these notations we have a local-global arithmetic square analogy.

\begin{proposition}[arithmetic square]
Let $X$ be a finite spectrum, then
\[\xymatrix@R-.3cm@C-.2cm@M+.1cm{ 
X \ar[rrr]^{\prod_p L_{\mathbb{Z}_{(p)}}}\ar[d]_{L_{\mathbb{Q}}}       & & & \prod_p X_{(p)}\ar[d]^{L_{\mathbb{Q}}}  \\
X_{\mathbb{Q}} \ar[rrr]_{L_{\mathbb{Q}}(\prod_p L_{\mathbb{Z}_{(p)}})} & & & (\prod_p X_{(p)})_{\mathbb{Q}}    
}\]
is a homotopy pullback square.
\end{proposition}

\subsubsection{$K$-theoretic localization of spectra}
Following section 8 of \cite{Rav84} we have
\begin{theorem}
Let $K$ and $KO$ be the spectra representing complex and real $K$-theory, respectively. Then $K\wedge X=pt$ if and only if $KO\wedge X=pt$, so $L_K$ and $L_{KO}$ represent the same functors and we have the same Bousfield classes $\langle K\rangle=\langle KO\rangle$.
\end{theorem}
\begin{proof}
Consider $\mathbb{CP}^2= S^2\cup_{\eta} e^4$ with the Hopf map $\eta: S^3\rightarrow S^2$ being the attaching map of $e^4$. Let us denote by $S^0\cup_{\eta} e^2$ the suspension spectrum with $\mathbb{CP}^2$ being the second suspension. Due to Adams we have
$$KU=KO\wedge \mathbb{CP}^2.$$
Hence having the cofiber sequence
$$ S^3\stackrel{\eta}{\rightarrow} S^2\rightarrow \mathbb{CP}^2 $$
we get by smashing with $KO$
$$\Sigma^3 KO \stackrel{\Sigma^2 \eta^{st}}{\longrightarrow} \Sigma^2 KO\rightarrow KO\wedge \mathbb{CP}^2=\Sigma^2 KU$$
and get by applying $\Sigma^{-2}$ in the stable homotopy category the famous cofiber sequence
$$ \Sigma KO\stackrel{\eta}{\rightarrow} KO \rightarrow KU.$$
Looking at this sequence $KO\wedge X=pt$ implies $K\wedge X=pt$ and conversely if $K\wedge X=pt$ then $\eta$ induces an automorphism of $KO_*X$. But since $\eta$ is nilpotent ($\eta^4=0$) we have $KO_*X=0$. 
\end{proof}

\subsection{Algebraic manipulations of spectra}
When working with ring spectra algebraically one often looks at them as algebraic rings. And as the study of rings is often simplified by passage to its quotients and localizations one also wants to transfer these techniques to ring spectra. As in \cite{HS} we want to recall the construction of quotients and localizations. Suppose that $E$ is a ring spectrum and that $\pi_*E=R$ is commutative. Given $x\in R$, define the spectrum $E/(x)$ by the cofibration
$$ \Sigma^{|x|} E \stackrel{x\cdot}{\rightarrow} E \rightarrow E/(x). $$
If $x$ is a non-zero divisor then $\pi_* E/(x)$ is isomorphic to the ring $R/(x)$ and in "good" cases $E/(x)$ is a ring spectrum and the map $E\rightarrow E/(x)$ is a map of ring spectra. Such a "good" situation is given if one has a regular sequence $\{x_1,x_2,...,x_n,...\}\subset R$ in which one can iterate the above situation and form a ring spectrum $E/(x_1,x_2,...,x_n,...)$ with
$$ \pi_* E/(x_1,x_2,...,x_n,...) \cong R/(x_1,x_2,...,x_n,...)$$
such that the natural map
$$ E\rightarrow E/(x_1,x_2,...,x_n,...)$$
is a map of ring spectra.
Considering the case of localizations, suppose that $S\subset R$ is a closed subset. Since $S^{-1}R$ is a flat $R$-module, the functor
$$ S^{-1} R \otimes_R E_*(-)$$
is a homology theory denoted by $S^{-1}E$. In "good" cases it is represented by a ring spectrum, and the localization can be described by a map of ring spectra
$$ E\rightarrow S^{-1} E.$$
Now we want to apply the above constructions to the Brown-Peterson spectrum $BP$ (confer \cite{Rav84}). In this case everything is "good" and all the constructions can be made. Recall that 
$$ BP_* \cong \mathbb{Z}_{(p)}[v_1,...,v_n,...] \quad\text{ with }\quad |v_n|=2p^n-2.$$
For $0<n<\infty$ the Morava $K$-theory ring spectra $K(n)$ and the Johnson-Wilson ring spectra $E(n)$ are defined by the isomorphisms
\begin{eqnarray*}
K(n)_* &\cong & \mathbb{F}_p[v_n,v_n^{-1}]\\
E(n)_* &\cong & \mathbb{Z}_{(p)}[v_1,...,v_n,v_n^{-1}]
\end{eqnarray*}
with the understanding that they are constructed from $BP$ using a combination of the above methods.

\subsubsection{The chromatic framework}
In the following paragraph we want to describe the chromatic framework which gives us the motivation for $K(1)$-localization.
\begin{theorem}[chromatic convergence theorem]
Let $p$ be a fixed prime and denote localization with respect to the Johnson-Wilson theories $E(n)_*$ \textup{(}$n\ge 0$\textup{)} by $L_n$. Then there are natural maps $L_n X\rightarrow L_{n-1} X$ for all spectra $X$, and if $X$ is a $p$-local finite spectrum, then the natural map
$$ X \rightarrow \holim L_n X$$
is a weak equivalence.
\end{theorem}
Recall from of \cite[p.\ 361]{Rav84} that the spectra $E(n)$ and $K(0)\vee K(1)\vee ... \vee K(n)$ have the same Bousfield classes. As said above there are natural transformations $L_n\rightarrow L_{n-1}$ and compatible transformations $1\rightarrow L_n$ giving the so-called {\em chromatic tower}
\[\xymatrix@R-.3cm@C-.2cm@M+.1cm{
                      &   & \vdots\ar[d]\\
                      &   & L_2 S\ar[d] \\
                      &   & L_1 S\ar[d]\\
 S\ar[rr]\ar[urr]\ar[uurr] \ar[uuurr] & & L_0 S
}\]
Looking at the tower we are interested in the difference between $L_n$ and $L_{n-1}$. On the one hand the fiber $M_n$ of the transformation $L_n\rightarrow L_{n-1}$ is known as the monochromatic layer, and on the other hand the difference of $L_n$ and $L_{n-1}$ is measured by the functor $L_{K(n)}$, which is localization with respect to the $n^{th}$ Morava $K$-theory. From \cite{HG94} we cite that there are natural equivalences
$$ L_{K(n)} M_n F \cong L_{K(n)} F\quad\text{ and } \quad M_nL_{K(n)}F\cong M_n F,$$
so the homotopy types of $L_{K(n)}F$ and $M_n F$ determine each other.
\begin{theorem}[arithmetic square]
Let $K(n)_*$ denote the $n^{th}$ Morava $K$-theory. There is a natural commutative diagram
\[\xymatrix@R-.3cm@C-.2cm@M+.1cm{
L_n X\ar[d]\ar[r] & L_{K(n)} X\ar[d]\\
L_{n-1} X\ar[r]   & L_{n-1} L_{K(n)} X
}\]
which for any spectrum $X$ is a homotopy pullback square. 
\end{theorem}

Informally one might say that, having a $p$-local spectrum $X$, the basic building blocks for the homotopy type of $X$ are the Morava $K$-theory localizations $L_{K(n)} X$. Comparing the stable homotopy category with the integers in arithmetic, it is the localization functors $L_{K(n)}$ which take over the role of the primes. 

\subsection{A resolution of the $K(1)$-local sphere}
In this section we are going to introduce a very useful $K(1)$-local fiber sequence which comes up from the resolution of the $K(1)$-local sphere and which will help us constructing an Artin-Schreier class and in the construction of spherical classes. Although we only need this special fiber sequence, we will sketch quite briefly also the general setup for the resolution of the $K(n)$-local sphere, which will give us some motivation for all that stuff: It is the theory of formal group laws which is the starting point. We refer to \cite{Re97}: Having the Honda formal group law of height $n$ -- which is characterized by its $p$-series $[p]_{\Gamma_n}(x)=x^{p^n}$ -- we apply the theory of Lubin-Tate deformation theories which is Landweber exact and gives homology theories called {\em Morava E-theories} $E_n$. Considering the automorphisms $\Aut(\Gamma_n)$ of the Honda formal group law $\Gamma_n$ (also known as the Morava stabilizer group $\mathbb{S}_n$) one considers the group
$$ \mathbb{G}_n=\Aut(\Gamma_n) \rtimes \Gal(\mathbb{F}_{p^n}/\mathbb{F}_p)$$
which gives a group action on $E_{n*}$. The Hopkins-Miller theorem states that $\mathbb{S}_n$ gives an action on the spectrum $E_n$ itself, and the Adams-Novikov spectral sequence 
$$ E_2^{s,t}:=H^s(\mathbb{S}_n,(E_n)_t)^{\Gal(\mathbb{F}_{p^n}/\mathbb{F}_p)} \Rightarrow \pi_{t-s} L_{K(n)} S^0$$
provides computational methods for calculating the homotopy groups of the $K(n)$-local sphere.

\subsubsection{The case $n=1$}
For $n=1$ the Honda formal group law $\Gamma_1$ coincides with the multiplicative formal group law $\mathbb{G}_m(x,y)=x+y+xy=(1+x)(1+y)-1$ because their $p$-series coincide modulo $p$
$$ [p]_{\mathbb{G}_m}(x)=(1+x)^p -1 \equiv x^p = [p]_{\Gamma_1}(x).$$
In our $p$-local setting we have $\Aut(\mathbb{G}_m)\cong \mathbb{Z}_p^{\times}$ and thus we have
$$ \mathbb{G}_1 =\mathbb{S}_1 \cong \mathbb{Z}_p^{\times},$$
where $\mathbb{G}_1\cong \mathbb{Z}_p\times C_{p-1}$ for $p$ odd and $\mathbb{G}_1\cong \mathbb{Z}_2\times C_2$ for $p=2$. Following \cite{GHMR} there is a short exact sequence of continuous $\mathbb{G}_1$-modules
$$ 0\rightarrow \mathbb{Z}_p[\![\mathbb{G}_1/F]\!]\rightarrow \mathbb{Z}_p[\![\mathbb{G}_1/F]\!]\rightarrow \mathbb{Z}_p\rightarrow 0$$
where $F$ is the maximal finite subgroup of $\mathbb{G}_1$. These resolutions of the trivial module are analogues of the fibrations
$$ L_{K(1)} S^0\cong E_1^{h\mathbb{G}_1}\rightarrow E_1^{hF}\rightarrow E_1^{hF}$$
with the notation meaning the homotopy fixed point spectra with respect to the given group. We note that $p$-adic complex $K$-theory $K\mathbb{Z}_p$ is a model for $E_1$ and the homotopy fixed point spectrum $E_1^{hC_2}$ can be identified with 2-adic real $K$-theory $KO\mathbb{Z}_2$. Since every $p$-adic unit $k\in \mathbb{Z}_p^{\times}$ gives an Adams operation $\psi^k$ and vice versa, and for example 3 is a topolocal generator for $\mathbb{Z}_2^{\times}/\{\pm 1\}$, we can write
$$ L_{K(1)} S^0 \rightarrow KO\mathbb{Z}_2 \stackrel{\psi^3-1}{\rightarrow} KO\mathbb{Z}_2.$$ 
This is a resolution of the $K(1)$-local sphere -- and the $K(1)$-local sphere is the fiber of the Adams operation $\psi^3-1$. Hence this is our desired fiber sequence.
\begin{remark}\rm
At this point we want to point out that the profinite group $\mathbb{G}_1/F\cong \mathbb{Z}_p$ has the pro-group ring $\mathbb{Z}_p[\![\mathbb{G}_1/F]\!]$ which is isomorphic to the power series ring $\Lambda=\mathbb{Z}_p[\![T]\!]$ called the {\em Iwasawa algebra}. The Iwasawa algebra plays an important role in number theory when studying $\mathbb{Z}_p$-extensions of number fields and has been studied by many people for a long time. We refer to \cite{Wa97} for number theoretic properties and to \cite{HM07} for the interpretation for $K(1)$-local spectra. In \cite{HM07} the resolution of the $K(1)$-local sphere is stated as
$$ S^0 \rightarrow KO \stackrel{T}{\rightarrow} KO. $$
for $p=2$. Maybe more topology could be deduced from the analysis of the Iwasawa algebra.
\end{remark}

\subsection{Thom isomorphism}
Having a complex vector bundle $\xi:E\rightarrow X$ with $X$ compact Hausdorff one can construct the Thom space $\Th(\xi)=D(\xi)/S(\xi)$. If the bundle $\xi$ admits a Thom class $\tau\in \tilde{K}^*(\Th(\xi))$ we get an isomorphism of modules
$$ K^*(X) \stackrel{\cong}{\rightarrow} \tilde{K}^*(\Th(\xi))$$
called the Thom isomorphism. $\tilde{K}^*(\Th(\xi))$ is the free $K^*(X)$-module with single basis element the Thom class $\tau\in\tilde{K}^*(\Th(\xi))$.
\subsubsection{Generalization}
Let $E^*$ be a generalized multiplicative cohomology theory. 
\begin{definition}
A class $\tau\in E^*(D(\xi),S(\xi))$ is said to be a Thom class for $\xi$ if for every $x\in X$ the restriction of $\tau$ to $E^*(D(\xi_x),S(\xi_x))$ is an $E^*(pt)$-module generator. 
\end{definition}
Having a Thom class $\tau\in E^d(D(\xi),S(\xi))$ the homotopy equivalence $p:D(\xi)\rightarrow X$ inducing an isomorphism $p^*:E^*(X)\stackrel{\cong}{\rightarrow} E^*(D(\xi))$ leads us to the definition of the Thom isomorphism
$$ E^*(X) \rightarrow E^{*+d}(D(\xi),S(\xi))\cong \tilde{E}^{*+d}(\Th(\xi))$$
by applying the cup product and mapping
$$ \alpha \mapsto p^*(\alpha)\cup \tau.$$
This gives an isomorphism of graded modules over $E^*(pt)$.
\begin{remark}
For a trivial bundle of dimension 1 the Thom isomorphism reduces to the suspension isomorphism.
\end{remark}
To prove the Thom isomorphism for compact $X$ one proceeds by induction over the open sets in a trivialization of $\xi$ using the suspension isomorphism as the starting case and the Mayer-Vietoris sequence to carry out the inductive step.
\begin{remark}
There is also a homology Thom isomorphism
$$ \tilde{E}_{*+d}(\Th(\xi)) \stackrel{\cong}{\rightarrow} E_*(X), $$
using the cap product with the Thom class
$$ \cap: \tilde{E}_{*+d}(\Th(\xi)) \times \tilde{E}^d(\Th(\xi)) \rightarrow E_*(X)$$
instead of the cup product.
\end{remark}

\subsubsection{Strong form of the Thom isomorphism}
Following along the lines of \cite{MR81} one can state a strong form of the Thom isomorphism theorem. In this context we assume given a stable spherical fibration $\nu: X\rightarrow BF$ over a locally finite CW complex $X$, and a ring spectrum $E$ orienting $\nu$, i.e.\ a Thom class $\tau: \Th(\nu)\rightarrow E$ whose restriction to a fiber $S^0\hookrightarrow \Th(\nu)$ is the unit of $E$.
\begin{theorem}[Mahowald, Ray] 
There is a homotopy equivalence
$$ \alpha(\tau): E\wedge \Th(\nu) \rightarrow E\wedge X_+ $$
which on homotopy groups induces the traditional Thom isomorphism
$$ \phi_{\tau}: E_*(\Th(\nu))=\pi_*(E\wedge \Th(\nu)) \stackrel{\alpha(\tau)_*}{\rightarrow} \pi_*(E\wedge X_+)=E_*(X_+).$$
\end{theorem}
\begin{proof}
First suppose that $X$ has finite dimension, so that $\nu$ lifts to $\nu_n: X\rightarrow BF_n$ for suitably large $n$. Let $p_n: S(\nu_n)\rightarrow X$ be the associated $n$-sphere fibration, so that $\Th(\nu_n)=X\cup_{p_n} CS(\nu_n)$. Now we define the Thom diagonal 
$$\Delta: \Th(\nu_n)\rightarrow \Th(\nu_n)\wedge X_+$$
to be
$$ \Delta(x)=\begin{cases} (x,p_n(x)) & \text{ for } x\neq \infty\\ \infty & \text{ for } x=\infty \end{cases} $$
and consider the composite
$$\alpha(\tau): E\wedge \Th(\nu_n)\stackrel{\id\wedge \Delta}{\rightarrow} E\wedge \Th(\nu_n)\wedge X_+\stackrel{\id\wedge \tau\wedge \id }{\rightarrow} E\wedge \Sigma^{n+1} E\wedge X_+\stackrel{\mu\wedge\id}{\rightarrow} E\wedge \Sigma^{n+1} X_+$$
where $\mu$ is the product in $E$. On homotopy groups this map induces a homomorphism
$$\phi_{\tau}: E_{*+n+1}(\Th(\nu_n))\rightarrow E_*(X_+)$$
which is the usual homology Thom isomorphism, i.e.\ cap product with $\tau$.
\end{proof}
The above theorem suffices for our purposes, i.e. we have $K_*MU\cong K_*BU$ and $K_*MSU\cong K_*BSU$.
\begin{example}
To define the Thom isomorphism $\Phi:K_*MU\rightarrow K_*BU$ we use the Thom diagonal 
$$MU\stackrel{\Delta}{\rightarrow} BU\wedge MU,$$ 
and choose a Thom class $MU\stackrel{\tau}{\rightarrow} K$. Thus we define the Thom isomorphism: An element $\overline{f}\in K_nMU$ represented by $S^n\stackrel{f}{\rightarrow} MU\wedge K$ is mapped to $\Phi \overline{f}$, more explicitly
$$S^n\stackrel{f}{\rightarrow} MU\wedge K\stackrel{\Delta\wedge 1}{\longrightarrow} BU\wedge MU\wedge K\stackrel{BU\wedge \tau\wedge 1}{\longrightarrow} BU\wedge K\wedge K \stackrel{BU\wedge \mu}{\longrightarrow} BU\wedge K.$$
\end{example}

%
%
\section{Splitting off an $E_{\infty}$ summand $T_{\zeta}$}
In the $K(1)$-local world at the prime $p=2$, we take the fiber sequence $S\rightarrow KO \stackrel{\psi^3-1}{\rightarrow} KO$ and look at the homotopy long exact sequence
$$... \rightarrow \pi_0 S^0 \longrightarrow KO_0 \stackrel{\psi^3-1}{\longrightarrow} KO_0 \longrightarrow \pi_{-1} S^0 \rightarrow ...$$
Since $KO_0\cong \mathbb{Z}_2$ are the 2-adic integers and $\psi^3$ is a ring homomorphism, $\psi^3-1$ is the zero map on $KO_0$. Thus $KO_0\rightarrow \pi_{-1}S^0$ is injective and the image of $1$ is a non-trivial element $\zeta\in \pi_{-1}S^0\cong \mathbb{Z}_2$. Now we are attaching a $0$-cell along $\zeta$ and take the homotopy pushout in the category of $E_{\infty}$ spectra:
\[
\xymatrix@R-.3cm@C-.2cm@M+.1cm
{ S^{-1} \ar[d]^{*}\ar[r]^{\zeta} & S^0 & & TS^{-1} \ar[d]^{T*}\ar[r]^{\zeta} & S^0 \ar[d]\\
D^0 &                                        & & TD^0=S^0 \ar[r] & T_{\zeta}
}\]
This $E_{\infty}$ spectrum $T_{\zeta}$ will be an $E_{\infty}$ summand in $MSU$. For this
\[
\xymatrix@R-.3cm@C-.2cm@M+.1cm
{ TS^{-1} \ar[d]^{T*}\ar[r]^{\zeta} & S^0 \ar[d]\ar@/^/[ddr] &\\
TD^0=S^0 \ar[r]\ar@/_/[drr] & T_{\zeta} \ar@{.>}[dr] & \\
 & & MSU
}\]
we have to show that $\zeta\in \pi_{-1} MSU$ vanishes. Considering the diagram
\[
\xymatrix@R-.3cm@C-.2cm@M+.1cm
{    & KO_0 S^0 \ar[r]\ar[d] & \pi_{-1} S^0 \ar[d]\\
KO_0 MSU \ar[r]^{\psi^3-1} & KO_0 MSU \ar[r] & \pi_{-1} MSU
}\]
it is sufficient to find an element $b\in KO_0 MSU$ mapping to $1$, because on the one hand the element $1\in KO_0S^0$ maps to $1\in KO_0MSU$ going to $0\in \pi_{-1} MSU$ due to the long exact sequence, and on the other hand the element $1\in KO_0S^0$ maps to $\zeta\in \pi_{-1}S^0$, which has to vanish in $\pi_{-1}MSU$ because the diagram commutes.
\begin{definition}
An Artin-Schreier class is a class $b\in KO_0 MSU$ with $\psi^3 b=b+1$.
\end{definition}
In the following part we construct such a class rationally and then give a construction of an $SU$-manifold which realizes this class.

\subsection{The image of $MSU_*\rightarrow MU_*$}
In $MSU_*$ every torsion is 2-torsion which is the kernel of $MSU_*\rightarrow MU_*$ concentrated in dimensions $8k+1$ and $8k+2$ for $k\ge 0$; in these cases $MSU_{8k+1}\cong MSU_{8k+2}$ is an $\mathbb{F}_2$ vector space whose dimension is the number of partitions of $k$ (compare \cite{CF66b}). Due to a theorem by Thom, complex bordism is rationally represented by complex projective spaces:
\begin{theorem}[Thom] 
$$MU_*\otimes \mathbb{Q} = \mathbb{Q}[\mathbb{CP}^n|n\ge 1] .$$
\end{theorem}
\noindent
The obstruction for a $U$-manifold to be an $SU$-manifold is the first Chern class $c_1$ of the tangent bundle. Hence a manifold $M\in MSU_4$ is rationally a linear combination
$$ M=A\cdot \mathbb{CP}^1\times \mathbb{CP}^1 + B\cdot \mathbb{CP}^2\quad\text{with}\quad c_1^2[M]=0;$$
in the above notation we always mean their bordism classes and have omitted the brackets for brevity.
An example of an $SU$-manifold is the Kummer surface
$$\mathcal{K}=K3=\{z\in \mathbb{CP}^3 | z_0^4+z_1^4+z_2^4+z_3^4=0\}$$
which is $U$-bordant to $ K3 \sim_U 18 (\mathbb{CP}^1)^2 - 16\mathbb{CP}^2$. Indeed $MSU_4=\mathbb{Z}\langle K3 \rangle$ since the Todd-genus ($\hat{A}$-genus respectively) of an $SU$-manifold is even and $Td(K3)=2$. It turns out that we cannot construct an Artin-Schreier class out of a class in $MSU_4$ since we need an $SU$-manifold with $\hat{A}=1$. Therefore we are interested in the image of $MSU_8 \rightarrow MU_8$. Rationally this is a linear combination
$$ M= A\cdot \mathbb{CP}^4 + B\cdot \mathbb{CP}^1\times \mathbb{CP}^3 + C\cdot (\mathbb{CP}^2)^{\times 2} + D\cdot (\mathbb{CP}^1)^{\times 4} + E\cdot (\mathbb{CP}^1)^{\times 2} \times \mathbb{CP}^2;$$
requiring the first Chern class to vanish implies the conditions $ c_1^4[M]=c_1c_3[M]=c_1^2 c_2[M]=0$ in the Chern numbers. To express them as linear equations in the coefficients we first have to calculate the total Chern classes of the complex projective spaces and their products:
\begin{eqnarray*}
c(T\mathbb{CP}^4) & = & c(1\oplus T\mathbb{CP}^4)=c(5L^*)=(1+x)^5=1+5x+10x^2+10x^3+5x^4\\
c(T(\mathbb{CP}^1\times\mathbb{CP}^3)) & = & pr_1^*c(T\mathbb{CP}^1)\cdot pr_2^* c(T\mathbb{CP}^3) =(1+x_1)^2(1+x_2)^4\\ &=&(1+2x_1)(1+4x_2+6x_2^2+4x_2^3)\\
& = & 1+(2x_1+4x_2)+(8x_1x_2+6x_2^2)+(12x_1x_2^2+4x_2^3)+8x_1x_2^3\\
c(T(\mathbb{CP}^2\times\mathbb{CP}^2)) & = & pr_1^* c(T\mathbb{CP}^2)\cdot pr_2^* c(T\mathbb{CP}^2)=(1+x_1)^3(1+x_2)^3\\
& = & (1+3x_1+3x_1^2)(1+3x_2+3x_2^2)\\
& = & 1+(3x_1+3x_2)+(3x_1^2+9x_1x_2+3x_2^2) +(9x_1^2x_2+9x_1x_2^2)+9x_1^2x_2^2\\
c(T(\mathbb{CP}^1)^{\times 4}) & = & (1+x_1)^2(1+x_2)^2(1+x_3^2)(1+x_4)^2 \\
&=&(1+2x_1)(1+2x_2)(1+2x_3)(1+2x_4)\\
&=& 1+ 2(x_1+x_2+x_3+x_4)\\
& & +4(x_1x_2+x_1x_3+x_1x_4+x_2x_3+x_2x_4+x_3x_4)\\
& & +8(x_1x_2x_3+x_1x_2x_4+x_1x_3x_4+x_2x_3x_4)+16x_1x_2x_3x_4\\
c(T((\mathbb{CP}^1)^2\times \mathbb{CP}^2)) & = & (1+x_1)^2(1+x_2)^2(1+x_3)^3=(1+2x_1)(1+2x_2)(1+3x_3+3x_3^2)\\ 
&=& 1 + (2x_1+2x_2+3x_3) + (4x_1x_2+6x_1x_3+6x_2x_3+3x_3^2)\\
& & + (6x_1x_3^2+6x_2x_3^2+12x_1x_2x_3)+12x_1x_2x_3^2
\end{eqnarray*}
Now we calculate the Chern numbers $c_1^4(TM)[M]$, $c_1c_3(TM)[M]$ and $c_1^2c_2(TM)[M]$ by evaluating them on the complex projective spaces:
\begin{eqnarray*}
c_1^4(T\mathbb{CP}^4)[\mathbb{CP}^4] &=& (5x)^4[\mathbb{CP}^4]=625\\
c_1^4[\mathbb{CP}^1\times\mathbb{CP}^3] &=&(2x_1+4x_2)^4[\mathbb{CP}^1\times\mathbb{CP}^3]=512x_1x_2^3[\mathbb{CP}^1\times\mathbb{CP}^3]=512\\
c_1^4[\mathbb{CP}^2\times\mathbb{CP}^2] &=& 3^4(x_1+x_2)^4[\mathbb{CP}^2\times\mathbb{CP}^2]=486x_1^2x_2^2[\mathbb{CP}^2\times\mathbb{CP}^2]=486\\
c_1^4[(\mathbb{CP}^1)^{\times 4}]&=&2^4(x_1+x_2+x_3+x_4)^4[(\mathbb{CP}^1)^{\times 4}]\\ & = & 2^4\cdot 4! \cdot x_1x_2x_3x_4[(\mathbb{CP}^1)^{\times 4}]=384\\
c_1^4[(\mathbb{CP}^1)^{\times 2}\times \mathbb{CP}^2)] &=& (2x_1+2x_2+3x_3)^4[(\mathbb{CP}^1)^{\times 2}\times \mathbb{CP}^2)]\\
&= & 432x_1x_2x_3^2[(\mathbb{CP}^1)^{\times 2}\times \mathbb{CP}^2)]=432
\end{eqnarray*}
gives the equation
$$c_1^4[M]=  0 = 625A + 512B + 486C + 384D + 432E,$$
evaluation of $c_1c_3(TM)[M]$
\begin{eqnarray*}
c_1c_3(T\mathbb{CP}^4)[\mathbb{CP}^4] &=& 5x\cdot 10x^3 [\mathbb{CP}^4]=50\\
c_1c_3[\mathbb{CP}^1\times\mathbb{CP}^3] &=&(2x_1+4x_2)(12x_1x_2^2+4x_2^3)[\mathbb{CP}^1\times\mathbb{CP}^3]=56x_1x_2^3[\mathbb{CP}^1\times\mathbb{CP}^3]=56\\
c_1c_3[\mathbb{CP}^2\times\mathbb{CP}^2] &=& (3x_1+3x_2)(x_1^3+9x_1^2x_2+9x_1x_2^2+x_2^3)[\mathbb{CP}^2\times\mathbb{CP}^2]\\
&=& 54x_1^2x_2^2[\mathbb{CP}^2\times\mathbb{CP}^2]=54\\
c_1c_3[(\mathbb{CP}^1)^{\times 4}] &=& 16(x_1+x_2+x_3+x_4)(x_1x_2x_3+x_1x_2x_4+x_1x_3x_4+x_2x_3x_4)[(\mathbb{CP}^1)^{\times 4}]\\
&=&64 x_1x_2x_3x_4[(\mathbb{CP}^1)^{\times 4}]=64\\
c_1c_3[(\mathbb{CP}^1)^{\times 2}\times \mathbb{CP}^2)] &=& (2x_1+2x_2+3x_3)(6x_1x_3^2+6x_2x_3^2+12x_1x_2x_3)[(\mathbb{CP}^1)^{\times 2}\times \mathbb{CP}^2)]\\
&=&60x_1x_2x_3^2[(\mathbb{CP}^1)^{\times 2}\times \mathbb{CP}^2)]=60
\end{eqnarray*}
gives the equation
$$c_1c_3[M]=0=50A + 56B+ 54C+ 64D + 60E,$$
and evaluation of $c_1^2c_2(TM)[M]$
\begin{eqnarray*}
c_1^2c_2(T\mathbb{CP}^4)[\mathbb{CP}^4] &=& (5x)^2\cdot 10 x^2[\mathbb{CP}^4]=250x^4[\mathbb{CP}^4]=250\\
c_1^2c_2[\mathbb{CP}^1\times\mathbb{CP}^3]&=& (2x_1+4x_2)^2(x_1^2+8x_1x_2+6x_2^2)[\mathbb{CP}^1\times\mathbb{CP}^3]\\
&=& 224x_1x_2^3[\mathbb{CP}^1\times\mathbb{CP}^3]=224\\
c_1^2c_2[\mathbb{CP}^2\times\mathbb{CP}^2]&=& (3x_1+3x_2)^2(3x_1^2+9x_1x_2+3x_2^2)[\mathbb{CP}^2\times\mathbb{CP}^2]\\
&=&216x_1^2x_2^2[\mathbb{CP}^2\times\mathbb{CP}^2]=216\\
c_1^2c_2[(\mathbb{CP}^1)^{\times 4}]&=&16(x_1+x_2+x_3+x_4)^2\times\\
& &(x_1x_2+x_1x_3+x_1x_4+x_2x_3+x_2x_4+x_3x_4)[(\mathbb{CP}^1)^{\times 4}]\\
&=&192x_1x_2x_3x_4[(\mathbb{CP}^1)^{\times 4}]=192\\
c_1^2c_2[(\mathbb{CP}^1)^{\times 2}\times \mathbb{CP}^2)] &=& (2x_1+2x_2+3x_3)^2\times\\
& & (4x_1x_2+6x_1x_3+6x_2x_3+3x_3^2)[(\mathbb{CP}^1)^{\times 2}\times \mathbb{CP}^2)]\\
&=& 204x_1x_2x_3^2[(\mathbb{CP}^1)^{\times 2}\times \mathbb{CP}^2)]=204
\end{eqnarray*}
gives the equation
$$c_1^2c_2[M]= 0= 250A + 224B + 216C+ 192D + 204E.$$
Hence we consider the system of linear equations
\begin{center}
\begin{tabular}{rcrcrcrcrcr}
$c_1^4[M]$=    & 0= & 625A & + & 512B & + & 486C & + & 384D & + & 432E\\
$c_1c_3[M]$=   & 0= &  50A & + &  56B & + &  54C & + &  64D & + &  60E\\
$c_1^2c_2[M]$= & 0= & 250A & + & 224B & + & 216C & + & 192D & + & 204E\\
\end{tabular}
\end{center}
which is integrally equivalent to the following system of homogeneous linear equations:
\begin{center}
\begin{tabular}{crcrcrcrcr}
0= & 25A & + & 8B &  &   &   &     &   &     \\
0= &     & + & 4B &  &   & + & 16D & + &  9E \\
0= &     &   &    & -&27C& + & 48D & + & 15E \\
\end{tabular}
\end{center}
The space of solutions is 2-dimensional. We know one solution $K3\times K3$, i.e.\ the square of the Kummer surface, having the parameter representation $$(A,B,C,D,E) = (0,0,256,324,-576)$$ or
$$ \mathcal{K}^2=K3\times K3 \sim_U 256 \mathbb{CP}^2\times\mathbb{CP}^2   +324(\mathbb{CP}^1)^{\times 4} -576 (\mathbb{CP}^1)^{\times 2}\times \mathbb{CP}^2.$$
Another independent solution is given in parameter representation as $(A,B,C,D,E) = (8,-25,-12,-23,52)$ or as
$$N:=8\mathbb{CP}^4 -25\mathbb{CP}^1\times\mathbb{CP}^3 -12\mathbb{CP}^2\times\mathbb{CP}^2 -23 (\mathbb{CP}^1)^{\times 4} +52 (\mathbb{CP}^1)^{\times 2}\times \mathbb{CP}^2.$$
Hence we can rationally describe bordism classes of $SU$-manifolds under the injection $MSU_8\rightarrow MU_8$ via
$$ M= k\cdot (K3)^2+ l\cdot N$$
with $k,l\in\mathbb{Q}$. In the next section we take the values $(k,l)=(\frac{1}{4},12)$ and study its $K$-theory class under the map $MU_*\rightarrow K_*MU$ using Miscenkos formula which gives us an Artin-Schreier class. 

\subsection{Formal group laws and Miscenkos formula} 
\subsubsection*{Formal group laws} In the following part we briefly recall the notions of the theory of formal group laws which we use to construct the morphism $MU_*\rightarrow K_*MU$. We restrict to commutative, one-dimensional formal group laws.
\begin{definition}
Let $R$ be a commutative ring with unit. A formal group law over $R$ is a power series $F(x,y)\in R[\![x,y]\!]$ satisfying
\begin{enumerate}
\item $F(x,0)=x=F(0,x)$
\item $F(x,y)=F(y,x)$
\item $F(x,F(y,z))=F(F(x,y),z).$
\end{enumerate}
\end{definition}
These axioms correspond to the existence of a neutral element, commutativity and associativity in the group case. Obviously we can write $F(x,y)=x+y+\sum_{i,j\ge 1} a_{ij}x^iy^j$ with $a_{ij}=a_{ji}$, and in terms of the power series it is clear that there exists an inverse, i.e.\ a formal power series $\iota(x)\in R[\![x]\!]$ such that $F(x,\iota(x))=0$. Formal group laws are naturally related to complex oriented theories in the following way: The Euler class of a tensor product of line bundles defines a formal group law
$$\widehat{G_E}(x,y)=e(L_1\otimes L_2)\in E^*(\mathbb{CP}^{\infty}\times \mathbb{CP}^{\infty})\cong\pi_* E[\![x,y]\!]$$
with $x=e(L_1)$ and $y=e(L_2)$.
\begin{example}
The additive formal group law $\mathbb{G}_a(x,y)=x+y$ arises as an orientation of singular cohomology.
The multiplicative formal group law $\mathbb{G}_m(x,y)=x+y-xy$ comes up as an orientation of complex $K$-theory.
In the following we will encounter the universal formal group law $F_u$ via complex cobordism \textup{(}$MU$-theory \textup{)}.
\end{example}
\begin{definition}
Let $F$ and $G$ be formal group laws. A homomorphism $f:F\rightarrow G$ is a power series $f(x)\in R[\![x]\!]$ with constant term $0$ such that $f(F(x,y))=G(f(x),f(y))$. It is an isomorphism if it is invertible, i.e.\ if $f'(0)$ \textup{(}the coefficient of $x$\textup{)} is a unit in $R$, and a strict isomorphism if $f'(0)=1$. A strict isomorphism from $F$ to the additive formal group law $\mathbb{G}_a$ is called a logarithm for $F$, denoted $\log_F(x)$. Its inverse power series is called exponential, denoted $\exp_F(x)$.
\end{definition}

\begin{example}
Over a $\mathbb{Q}$-algebra every formal group law is isomorphic to the additive formal group law. Especially the logarithm of the universal formal group law is given by
$$ \log_{MU} (x)=\sum_{n\ge 0} \frac{[\mathbb{CP}^n]}{n+1} x^{n+1}.$$
\end{example}

\begin{proposition} If $x_1,x_2$ are two complex orientations for $E^*(-)$, then their associated formal group laws $F_1$ and $F_2$ are isomorphic.
\end{proposition}
In the context of formal group laws let $F_{MU}$ denote the universal formal group law
$$F_{MU}(x,y)=x+y+\sum_{i,j\ge 1} a_{ij} x^i y^j$$
with the coefficients $a_{ij}\in L$ in the Lazard ring with degree $|a_{ij}|=2-2(i+j)$. Let
$$F_K(x,y)=x+y+vxy$$
denote the multiplicative formal group law corresponding to the $K$-theory spectrum with $v$ the inverse Bott element with $|v|=-2$.
Now we are going to construct a morphism
$$ f: MU_*\rightarrow K_*MU$$
such that the induced formal group law 
$$ f^* F_{MU}(x,y):=x+y+\sum_{i,j\ge 1} f(a_{ij}) x^i y^j$$
is the formal group law $F_K$ twisted by the invertible power series 
$$g(x)=\sum_{i\ge 0} b_i x^{i+1}$$ 
(with $b_0=1$) defined by
$$ ^gF_K(x,y):=g(F_K(g^{-1}(x),g^{-1}(y)))=g(g^{-1}(x)+g^{-1}(y)+vg^{-1}(x)g^{-1}(y))$$
with $g^{-1}(g(x))=x$ the inverse function.

\subsubsection*{Boardman homomorphism}
The element $a_{ij}\in \pi_{2(i+j-1)}$ can be represented by a weakly almost complex manifold. To ask for the (normal) characteristic numbers of this manifold is (essentially) equivalent to asking for the image of $a_{ij}$ under the Hurewicz homomorphism
$$ \pi_*MU \rightarrow H_*MU.$$
We introduce the Boardman homomorphism, which is (slightly) more general than the Hurewicz homomorphism. Let $E$ be a (commutative) ring spectrum, then for any (space or spectrum) $Y$ we consider the map
$$ Y\cong S^0\wedge Y \stackrel{i\wedge 1}{\rightarrow} E\wedge Y.$$
Composing a map $X\rightarrow Y$ with this map induces a homomorphism
$$ B: [X,Y]_*\rightarrow [X,E\wedge Y]_* $$
called the Boardman homomorphism. The Hurewicz homomorphism is recovered by setting $X=S^0$ and $E=H$ (the Eilenberg-MacLane spectrum representing singular homology).\\ 
Since $E\wedge Y$ is at least a module spectrum over the ring spectrum $E$, we may obtain information about $[X,E\wedge Y]_r=(E\wedge Y)^{-r}(X)$ from $E_*(X)$, for example there is a universal coefficient theorem
\[\xymatrix@R-.3cm@C-.2cm@M+.1cm
{ [X,Y]_* \ar[rr]^B \ar[dr]_{\alpha} & & [X,E\wedge Y]_* \ar[dl]^p\\
 & \Hom_{\pi_*E}(E_*X,E_*Y) & }\]
where $\alpha(f)=f_*:E_*X\rightarrow E_*Y$ is the induced map in $E$-homology and $p$ is defined by $(p(h))(k)=\langle h,k\rangle \in E_*Y$ using the Kronecker pairing
$$ (E\wedge Y)^*(X) \otimes E_*X \rightarrow E_*Y $$
with 
$$h\otimes k \mapsto \langle h,k\rangle : S\rightarrow E\wedge X \stackrel{1\wedge h}{\rightarrow} E\wedge E\wedge Y \stackrel{\mu\wedge 1}{\rightarrow} E\wedge Y .$$

\subsubsection*{Miscenkos formula}
We recall that power series of the form $g(x)=x+b_1x^2+b_2x^3+...$ are strict isomorphisms 
$$ g: F \stackrel{\cong}{\longrightarrow} \,^gF=g(F(g^{-1}x,g^{-1}y))$$
and want to give the explicit coefficients of the inverse power series $g^{-1}(x)=\sum_{i\ge 0} c_i x^{i+1}$. We calculate the first coefficients taking everything modulo $x^6$ and using the identity
\begin{eqnarray*}
x & \equiv & g^{-1}(g(x))= g(x)+c_1 g(x)^2+c_2 g(x)^3 + c_3 g(x)^4 + c_4g(x)^5+... \quad (\mod x^6)\\
  & \equiv & x +b_1x^2+b_2x^3+b_3x^4+b_4x^5\\
  &   & +c_1(x^2+2b_1x^3+(2b_2+b_1^2)x^4+(2b_3+2b_1b_2)x^5)\\
  &   & +c_2(x^3+3b_1x^4+(3b_2+3b_1^2)x^5+c_3(x^4+4b_1x^5)+c_4x^5
\end{eqnarray*}
Comparing coefficients gives the system of equations
\begin{eqnarray*}
0 &=& c_1+b_1\\
0 &=& c_2+2b_1c_1+b_2\\
0 &=& c_3+3b_1c_2+c_1(2b_2+b_1^2)+b_3\\
0 &=& c_4+4b_1c_3+c_2(3b_2+3b_1^2)+c_1(2b_3+2b_1b_2)+b_4
\end{eqnarray*}
resulting in
\begin{eqnarray*}
c_1 &=& -b_1\\
c_2 &=& 2b_1^2-b_2\\
c_3 &=& -5b_1^3+5b_1b_2-b_3\\
c_4 &=& 14b_1^4-21b_1^2b_2+6b_1b_3+3b_2^2-b_4.
\end{eqnarray*}
Applying the residue theorem of complex analysis proves the following (as done in \cite[p.\ 65 Prop.\ (7.5)]{Adams74} ):
\begin{proposition} Denoting the degree $2n$-part of an inhomogeneous polynomial with a lower index $n$ we have
$$c_n=\frac{1}{n+1}(\sum_{i\ge 0} b_i)^{-(n+1)}_n \text{ and }\quad b_n=\frac{1}{n+1}(\sum_{i\ge 0} c_i)^{-(n+1)}_n.$$
\end{proposition}

Next we explicitly calculate $^gF_K(x,y)=g(g^{-1}x+g^{-1}y+vg^{-1}x g^{-1}y):$
\begin{eqnarray*}
^gF_K(x,y)&=& x+y+ (v+2b_1)xy +(b_1v-2b_1^2+3b_2)(x^2y+xy^2)\\
          & & + (2vb_2-2vb_1^2+4b_3-8b_1b_2+4b_1^3) (x^3y+xy^3)\\
          & & + (v^2b_1-3vb_1^2+2b_1^3-6b_1b_2+6vb_2+6b_3)x^2y^2\\
          & & + (5vb_1^3 - 8vb_1b_2 + 25b_1^2b_2 + 3vb_3 -10b_1^4 -14b_1b_3 -6b_2^2 + 5b_4)\\
          & & \quad \times(x^4y+xy^4)\\
          & & + (4vb_1^3 - 18vb_1b_2 -4b_1^4 +8b_1^2b_2 - 2v^2b_1^2 + 3v^2b_2 -3b_2^2  \\
          & & \quad -16b_1b_3 + 12vb_3 +10b_4)\times (x^3y^2+x^2y^3)\\
          & & + \text{ higher order terms.}
\end{eqnarray*}
This implies:
\begin{eqnarray*}
a_{11}&\mapsto & v+2b_1\\
a_{21}&\mapsto & vb_1-2b_1^2+3b_2\\
a_{31}&\mapsto & 2vb_2 -2vb_1^2 + 4b_3 -8b_1b_2 + 4b_1^3\\
a_{22}&\mapsto & v^2b_1-3vb_1^2 +2b_1^3 - 6b_1b_2 + 6vb_2 + 6b_3\\
a_{41}&\mapsto & 5vb_1^3 - 8vb_1b_2 + 25b_1^2b_2 + 3vb_3 -10b_1^4 -14b_1b_3 -6b_2^2 + 5b_4\\
a_{32}&\mapsto & 4vb_1^3 - 18vb_1b_2 -4b_1^4 +8b_1^2b_2 - 2v^2b_1^2 + 3v^2b_2 -3b_2^2 -16b_1b_3 + 12vb_3 +10b_4
\end{eqnarray*}

Recall that the complex manifold $\mathbb{CP}^n$ defines an element $[\mathbb{CP}^n]\in\pi_{2n}MU$. The Hurewicz homomorphism
$$\pi_* MU \rightarrow H_* MU$$
tells us that the image of $[\mathbb{CP}^n]$ in $H_{2n}MU$ is $(n+1)c_n$ since the formula $(\sum_{i\ge 0} b_i)^{-(n+1)}_n$ gives the normal Chern numbers of $\mathbb{CP}^n$. The most important formula for us will be
$$[\mathbb{CP}^n]=(n+1)c_n=(\sum_{i\ge 0} a_{1i})^{-1}_n$$
leading to
\begin{center}
\begin{tabular}{|l|}
\hline
$[\mathbb{CP}^1]= -a_{11}$\\
$[\mathbb{CP}^2]= -a_{12}+a_{11}^2$\\
$[\mathbb{CP}^3]= -a_{13}-a_{11}^3 +2a_{11}a_{12}$\\
$[\mathbb{CP}^4]= -a_{14}+a_{11}^4+a_{12}^2+2a_{11}a_{13}$\\
\hline
\end{tabular}
\end{center}
Substituting these formulas we get
\begin{eqnarray*}
[N] & = & -112vb_1^3+340vb_1b_2 + 256b_1^2b_2-60vb_3 -184 b_1^4 + 40b_1b_3\\
    &   & +12b_2^2-40b_4+48v^2 b_2 + 58v^2 b_1^2 + 22 v^3b_1
\end{eqnarray*}
and
\begin{eqnarray*}
\frac{1}{4}[K3^2] & = & v^4 + 24 v^3 b_1 + 120v^2 b_1^2 +48v^2 b_2 -288vb_1^3+448 vb_1b_2\\
                  &   & +144b_1^4 - 576b_1^2+576b_2^2.
\end{eqnarray*}
Defining
$$ M:= \frac{1}{4} K3^2 + 12N$$
we get
\begin{eqnarray*}
[M] & = & v^4 + 16\cdot ( 18v^3b_1 + 51v^2 b_1^2+ 39v^2 b_2 -102vb_1^3 + 283 vb_1b_2 \\
    &   & - 45 vb_3 -129b_1^4 + 30b_1b_3 + 156b_1^2b_2 + 45b_2^2 - 30 b_4).
\end{eqnarray*}

\subsection{Construction of an $SU$-manifold with $\hat{A}=1$}
To split off the spectrum $T_{\zeta}$ from $MSU$ one essentially uses the existence of an Artin-Schreier class $b\in KO_0 MSU$ satisfying $\psi^3 b=b+1$. Via Miscenkos formula we have seen that such a class can be constructed with the logarithm construction if there is a Bott manifold whose associated $K$-theory class is congruent to $v^4$ modulo 16. Essentially we have to find a Bott manifold in $SU$ bordism, i.e.\ an $SU$-manifold $M$ with $\hat{A}([M])=1$ giving a periodicity element in $MSU_*$.

\subsubsection*{Main idea}
The Hopf bundle $\sigma: S^7\rightarrow S^4$ with fiber $S^3\cong SU(2)$ on the one hand admits an $SU$ structure and on the other hand generates $Im(J)_7\cong \Omega^{fr}_7\cong \pi_7^{st}\cong \mathbb{Z}/240$. Since $Td(D(\sigma))=1/240$ and since $240[\sigma]=0$ in $\Omega_7^{fr}$ implies the existence of a framed manifold $R^8$ with $\partial R^8=-240 \sigma$, we define
$$ B:= 240 D(\sigma) \cup_{240 \sigma} R^8$$
which serves as the desired Bott manifold, i.e.\ $Td(B)=\hat{A}(B)=1$.

\subsubsection*{$Sp(1)$-principal bundles over $S^4$}
With the identifications $ Sp(1)\cong SU(2) \cong S^3$ and $ Sp(2)/Sp(1) \cong S^7$ and 
$$ \frac{Sp(2)}{Sp(1)\times Sp(1)} \cong \mathbb{HP}^1 \cong S^4$$
we take the canonical $Sp(1)$-principal bundle over $S^4$
\[\xymatrix@R-.3cm@C-.2cm@M+.1cm
{ Sp(1)\cong S^3 \ar[r] & S^7 \ar[d]\\
 & S^4 }\]
i.e.\ the bundle whose associated line bundle $$ E:=S^7 \times_{Sp(1)} \mathbb{H}^1 \rightarrow S^4$$
satisfies $\langle c_2(E), [S^4] \rangle =1.$
We know that every $G$-principal bundle is given as the pullback of the universal $G$-principal bundle via the classifying map
\[\xymatrix@R-.3cm@C-.2cm@M+.1cm
{ f^*EG \ar[r]\ar[d] & EG \ar[d]\\
  B \ar[r]^f         & BG.}\]
In other words the functor $G\text{-}Pb(-)$ is representable by $BG$ and
$$ [B,BG]\cong G\text{-}Pb(B) \text{ via } f\mapsto f^*EG.$$
In the case of $Sp(1)$-principal bundles over $S^4$ we get
$$[S^4,BSp(1)]=[\Sigma S^3, BSp(1)] \cong [S^3,\Omega BSp(1)]=[S^3,Sp(1)]=[S^3,S^3] \cong \mathbb{Z}.$$
The canonical $Sp(1)$-principal bundle over $S^4$ is associated to $1\in\mathbb{Z}$.
We see that the disk bundle $Q:=D(E)$ with $\pi:Q\rightarrow S^4$ has as boundary 
$\partial Q= \partial D(E)=S(E)$ the original principal bundle.

\subsubsection*{Splitting of the tangent bundle $TQ$}
In general for a smooth vector bundle $\xi: E\rightarrow M$ the total space $E$ is again a smooth manifold. Now we are interested in the structure of the tangential bundle $TE$. There are two induced bundles, namely the induced tangential bundle and that of the total space:
\[\xymatrix@R-.3cm@C-.2cm@M+.1cm
{ \xi^* TM \ar[r]\ar[d] & E \ar[d]^{\xi}   &              & \xi^* E\ar[r]\ar[d] & E \ar[d]^{\xi}    \\
  TM \ar[r]             & M                & \text{and}   & E\ar[r]^{\xi}       & M        }\]
These already give an isomorphism
$$ TE \cong \xi^* TM \oplus \xi^*E. $$
Such a splitting of a tangent bundle is geometrically called a connection.
With the notation of above restricting the tangent bundle of the vector bundle to the disk bundle we get the splitting
$$ TQ \cong \pi^*E \oplus \pi^* TS^4;$$
note that the second summand is stably trivial. 

\subsubsection*{The Hopf bundle is an $SU$ manifold}
The Hopf bundle $\sigma: S^7\rightarrow S^4$ with fiber $S^3\cong SU(2)$ is not only an $SU(2)$-bundle but also an $SU$ manifold. A manifold $M$ has an $SU$ structure if its stable tangent bundle $TM$ is a complex vector bundle with a trivialization of its determinant bundle $det(TM)\cong 1_{\mathbb{C}}$.
\[\xymatrix@R-.3cm@C-.2cm@M+.1cm
{ D(\sigma) \ar[r]\ar[d] & \lambda_{taut} \ar[d]\ar[r] & \lambda_{taut}\ar[d]   \\
  S^4 \ar[r]             & BSU \ar[r]                  & BU.   }\]
From the splitting above we see the $SU$ structure, since the 8-dimensional bundle splits into two 4-dimensional bundles and $TS^4$ is stably trivial and $E$ is chosen to have vanishing~$c_1$.

\subsubsection*{Evaluation of the Todd genus}
We recall $Td= e^{c_1/2} \hat{A}$  and see that for $SU$ mani\-folds the Todd-genus and the $\hat{A}$-genus coincide. From \cite{Hi56} the degree 8-term of the Todd genus is given in Chern classes by:
$$ T_4=\frac{1}{720} (-c_4+c_3c_1+3c_2^2+4c_2c_1^2-c_1^4).$$
For the evaluation the Chern classes $c_1$ and $c_4$ do not contribute, due to the $SU$ structure and since $Q$ is a homotopy 4-sphere, respectively.
Next we emphasize that while for closed stably almost complex manifolds the Todd genus maps to the integers; the situation for $(U,fr)$ manifolds is different. A $(U,fr)$ manifold $M^n$ is a differentiable manifold $M$ with a given complex structure on its stable tangent bundle $TM$ and a given compatible framing of $TM$ restricted to the boundary $\partial M$. Their Chern numbers depend only on the bordism classes in $\Omega_n^{U,fr}$ and hence we have a Todd genus
$$ Td: \Omega^{U,fr}_{2n}\rightarrow \mathbb{Q}.$$
Moreover there is a commutative diagram
\[\xymatrix@R-.3cm@C-.2cm@M+.1cm
{ 0\ar[r] & \Omega^U_{2n} \ar[r]\ar[d]^{Td} & \Omega^{U,fr}_{2n} \ar[d]^{Td}\ar[r] & \Omega^{fr}_{2n-1} \ar[r]\ar[d]^{e_{\mathbb{C}}} & 0   \\
  0\ar[r] & \mathbb{Z} \ar[r]             & \mathbb{Q} \ar[r]                  & \mathbb{Q/Z}\ar[r] & 0   }\]
where $e_{\mathbb{C}}$ is the Adams $e$-invariant. This is worked out in \cite{CF66a}. As done on page 95 of \cite{CF66a} we can now evaluate the Todd genus 
\begin{eqnarray*}
\langle Td(TQ), [Q,\partial Q]\rangle &=& \langle \frac{1}{720} 3 c_2^2(E), [Q,\partial Q]\rangle =   \frac{1}{240} \langle c_2^2(E), [Q,\partial Q]\rangle\\
& =& \frac{1}{240}\langle c_2(E), [S^4]\rangle= \frac{1}{240}.
\end{eqnarray*}

\subsubsection*{Remark on the relation to $K$-theory}
In modern formulation the Todd genus is associated to the multiplicative formal group law and therefore to $K$-theory. Let $P(x)$ be a power series with 1 as constant coefficient. Its logarithm $g$ is given by
$$ g^{-1}(x)=\frac{x}{P(x)}.$$
Complex oriented cohomology theories always come with a formal group law $F(x,y)$ which can be expressed as
$$ F(x,y)=g^{-1}(g(x)+g(y)).$$
For the Todd genus we have $P(x)=\frac{x}{1-e^{-x}}$ implying $y=g^{-1}(x)=1-e^{-x}$. This gives us $g(x)=-\ln (1-x)$ and thus
\begin{eqnarray*}
F(x,y) & = & 1-\exp[-(-\ln(1-x)-\ln(1-y))]\\
       & = & 1-\exp(\ln(1-x)+\ln(1-y))\\
       & = & 1-(1-x)(1-y)=x+y-xy,
\end{eqnarray*}
which is the multiplicative formal group law coming from complex $K$-theory.

\subsubsection*{Definition of the Bott manifold}
Since $\partial Q=S^7$ is framed and $[\partial Q]\in \Omega_7^{fr}\cong \pi_7^s\cong \mathbb{Z}/240$ we have $240 [\partial Q]=0$, i.e.\ there exists a framed manifold $R^8$ with $\partial R^8=-240\partial Q$.
We define a Bott-manifold by
$$ B:= 240 Q \cup_{240 \partial Q} R^8$$
and see that indeed $\hat{A}(B)=Td(B)=240 Td(Q)+0=1$.

\subsection{Construction of an Artin-Schreier class}
Having a Bott manifold with associated $K$-theory class congruent to $v^4$ modulo $16$
we can use the power series of the logarithm 
$$\log (1+x)=\sum_{n=0}^{\infty} (-1)^n \frac{x^{n+1}}{n+1}$$
to define
$$ b= - \frac{\log([M])}{\log(3^4)}.$$
\begin{proposition}
The class $b$ is an Artin-Schreier class.
\end{proposition}
\begin{proof}
$$\psi^3 b= -\frac{\log([M]/3^4)}{\log(3^4)}=-\frac{\log([M])}{\log(3^4)} + \frac{\log(3^4)}{\log(3^4)}=b+1.$$
Here the stable Adams operation $\psi^k: K\rightarrow K$ is defined levelwise by $\frac{\Psi^k}{k^n}:K_{2n}\rightarrow K_{2n}$ with $\Psi^k$ being the unstable Adams operation. Inverting powers of $k\in\mathbb{Z}_2^{\times}$ is not a problem since everything is $2$-completed. 
\end{proof}

\subsection{Construction of an $E_{\infty}$ map $T_{\zeta}\rightarrow MSU$}
The fiber sequence $X\rightarrow KO\wedge X \stackrel{\psi^3-1}{\longrightarrow} KO\wedge X$ induces the unit map $\pi_0 KO\rightarrow \pi_{-1} S^0$ mapping $1\mapsto \zeta$. Now we define $T_{\zeta}$ to be the homotopy pushout in the category of $K(1)$-local $E_{\infty}$ ring spectra:
\[
\xymatrix@R-.3cm@C-.2cm@M+.1cm
{ TS^{-1} \ar[r]^(.45){T*}\ar[d]^{\zeta} & T*=S^0 \ar[d]\\
S^0 \ar[r] & T_{\zeta}
}\]
with $TX$ the free $E_{\infty}$ spectrum generated by the pointed space $X$. As the Hurewicz image of $\zeta\in\pi_{-1}MSU$ is zero we get a map $T_{\zeta}\rightarrow MSU$:
\[
\xymatrix@R-.3cm@C-.2cm@M+.1cm
{ TS^{-1} \ar[r]^(.45){T*}\ar[d]^{\zeta} & T*=S^0 \ar[d]\ar@/^/[ddr] &\\
S^0 \ar[r]\ar@/_/[drr] & T_{\zeta} \ar@{.>}[dr] & \\
 & & MSU
}\]

\subsection{Split map - direct summand argument}
To get $T_{\zeta}$ as a direct summand, one has to  construct a split $p$ such that the composition
$$ T_{\zeta}\stackrel{i}{\rightarrow} MSU \stackrel{p}{\rightarrow} T_{\zeta} $$
is the identity. This can be done using the $Spin$ splitting of Laures \cite{Laures03}
\[
\xymatrix@R-.3cm@C-.2cm@M+.1cm
{ T_{\zeta} \ar[d]^i \ar@/^1pc/[dr]  & \\
T_{\zeta}\wedge \bigwedge_{i=1}^{\infty} TS^0 \ar@/^1pc/[u]^p & MSpin \ar[l]_(0.4){\simeq}
}\]
and showing that the extended triangle commutes
\[
\xymatrix@R-.3cm@C-.2cm@M+.1cm
{ & MSU\ar[dd]^g\\ 
T_{\zeta} \ar[ur]^h \ar[d]^i \ar@/^1pc/[dr]  & \\
T_{\zeta}\wedge \bigwedge_{i=1}^{\infty} TS^0 \ar@/^1pc/[u]^p & MSpin \ar[l]_(0.4){\simeq}
}\]

\subsubsection{Comparison of the Artin-Schreier classes}
The $SU$ Artin-Schreier class constructed above is naturally also a $Spin$ Artin-Schreier class. Refering to \cite{Laures02} we have
\begin{lemma}
Let $b$ and $b'$ be two Artin-Schreier elements of $\pi_0 KO\wedge MSpin$. Then there is an $E_{\infty}$ self homotopy equivalence $\kappa$ of $MSpin$ which carries $b$ to $b'$.
\end{lemma}
\begin{proof}
The short exact sequence
$$ 0\rightarrow \pi_0 MSpin \rightarrow \pi_0 KO\wedge MSpin \stackrel{\psi^3-1}{\rightarrow} \pi_0 KO\wedge MSpin \rightarrow 0$$
with $(\psi^3-1)b=(\psi^3-1)b'=1$ tells us that $b$ and $b'$ can only differ by a class $a\in\pi_0MSpin$. Let $\kappa$ be the $E_{\infty}$ map of $$MSpin\cong T_{\zeta}\wedge \bigwedge TS^0$$ 
which is the identity on each $TS^0$ and restricts to
$$\iota+a\delta: C_{\zeta}\rightarrow MSpin$$
on $T_{\zeta}$. Then its inverse is defined in the same way with $a$ replaced by $-a$.
\end{proof}
With the notations $T_{\zeta}^{SU}$ and $T_{\zeta}^{Spin}$ for the $E_{\infty}$ spectra we get from the different Artin-Schreier classes, we have the following diagram with $E_{\infty}$ maps:
\[
\xymatrix@R-.3cm@C-.2cm@M+.1cm
{ T_{\zeta}^{SU}  \ar[r]\ar[d]^{\simeq}   & MSU\ar[dd] \ar@/^2pc/[l] \\ 
T_{\zeta}^{Spin}  \ar[d]^{\iota} \ar@/^1pc/[dr]  & \\
T_{\zeta}\wedge \bigwedge_{i=1}^{\infty} TS^0 \ar@/^1pc/[u]^{(id,*,*,...)} & MSpin \ar[l]_(0.4){\simeq}
}\]

%
%
\section{Detecting free $E_{\infty}$ summands $TS^0$}
\subsection{Introduction to Adams operations in $K$-theory}
\subsubsection{Basics}
In $K$-theory we have not only the ring structure, but also certain ring homomorphisms $\psi^k:K(X)\rightarrow K(X)$.
\begin{theorem}[Adams]
There exist ring homomorphisms $\psi^k: K(X)\rightarrow K(X)$, defined for all compact Hausdorff spaces $X$ and all integers $k\ge 0$, satisfying:
\begin{enumerate}
\item $\psi^k f^*=f^* \psi^k$ for all maps $f:X\rightarrow Y$ (naturality)
\item $\psi^k(L)=L^k$ if $L$ is a line bundle
\item $\psi^k\circ\psi^l=\psi^{kl}$
\item $\psi^p(\alpha)\equiv \alpha^p$ mod $p$ for $p$ prime
\end{enumerate}
\end{theorem}
At first we consider the special case when $E$ is a sum of line bundles $L_i$
$$\psi^k(L_1\oplus...\oplus L_n)=L_1^k + ... + L_n^k.$$
Next we are looking for a general definition of $\psi^k(E)$ which specializes to the above. We use the exterior powers $\lambda^i(E)$ to define
$$\lambda_t(E)=\sum_i \lambda^i(E) t^i\in K(X)[\![t]\!].$$
By naturality of the exterior power construction we have
\begin{eqnarray*}
\lambda_t(E_1\oplus E_2) & = & \sum_i \lambda^i(E_1\oplus E_2) t^i = \sum_i \bigoplus_{k=0}^i (\lambda^k(E_1)\otimes \lambda^{i-k}(E_2)) t^i\\
 & = & (\sum_i \lambda^i(E_1) t^i) \otimes (\sum_k \lambda^k(E_2) t^k) = \lambda_t(E_1)\otimes \lambda_t(E_2).
\end{eqnarray*}
In the case $E=L_1\oplus ... \oplus L_n$ we get
$$\lambda_t(E)=\prod_i \lambda_t(L_i) = \prod_i (1+L_i t)= 1+\sigma_1 + ... + \sigma_n;$$
thus comparing the coefficients, $\lambda^j(E)=\sigma_j(L_1,...,L_n)$ is the $j^{th}$ elementary symmetric polynomial. Hence we take the Newton polynomials $s_k$ and define
$$\psi^k(E)=s_k(\lambda^1(E),...,\lambda^k(E)),$$
which specializes to the formula above when $E$ is a sum of line bundles. Finally we prove that the stated properties are fulfilled.
\begin{proof}
The naturality $$f^*(\psi^k(E))=\psi^k(f^*(E))$$ follows from $f^*(\lambda^i(E))=\lambda^i(f^*(E))$. The additivity $\psi^k(E_1\oplus E_2)=\psi^k(E_1)+\psi^k(E_2)$ is done via the naturality and the splitting principle. Since $p:F(E)\rightarrow X$ induces an injection $p^*: K^*(X)\rightarrow K^*(F(E))$ we get
\begin{eqnarray*}
p^* \psi^k(E_1\oplus E_2) & = & \psi^k p^*(E_1\oplus E_2) = \psi^k(L_1\oplus ... \oplus L_m \oplus L_1'\oplus ... \oplus L_n')\\
  & = & \psi^k (p^*E_1 \oplus p^*E_2)=p^*\psi^kE_1 \oplus p^* \psi^k E_2.
\end{eqnarray*}
Since $\psi^k$ is additive on vector bundles, we get via the Grothendieck construction an additive operation on $K(X)$ defined by
$$\psi^k(E_1-E_2)=\psi^k(E_1)-\psi^k(E_2).$$
Next we prove multiplicativity: If $E$ is the sum of line bundles $L_i$ and $E'$ is the sum of line bundles $L_j'$ then $E\otimes E'$ is the sum of line bundles $L_i\otimes L_j'$ implying
\begin{eqnarray*}
\psi^k(E\otimes E') & = & \sum_{i,j} \psi^k(L_i\otimes L_j)=\sum_{i,j} (L_i\otimes L_j')^k\\
                    & = & \sum_{i,j} L_i^k\otimes L_j'^k=\sum_i L_i^k \sum_j L_j'^k=\psi^k(E)\psi^k(E').
\end{eqnarray*}
Thus $\psi^k$ is multiplicative on vector bundles and it follows formally that it is multiplicative on elements of $K(X)$. For $\psi^k\circ \psi^l=\psi^{kl}$ the splitting principle and additivity reduce to the case of line bundles where
$$\psi^k(\psi^l(L))=L^{kl}=\psi^{kl}(L).$$
Likewise for $\psi^p(\alpha)\equiv \alpha^p$ mod $p$, since for $E=L_1+...+L_n$ we have modulo $p$
$$\psi^p(E)=L_1^p+...+L_n^p\equiv (L_1+...+ L_n)^p=E^p.$$
\end{proof}
\subsubsection{Adams operations on the $K$-theory spectrum}
One recalls that Adams operations $\psi^k:K\rightarrow K$ are operations on the ring spectrum $K$ itself and that we have naturality, i.e.\ for every map $f:X\rightarrow Y$ the diagram
\[\xymatrix@R-.3cm@C-.2cm@M+.1cm{ 
 K\wedge X \ar[d]_{K \wedge f}\ar[r]^{\psi^k\wedge X} & K\wedge X \ar[d]^{K\wedge f}\\
 K\wedge Y \ar[r]^{\psi^k\wedge Y}                    & K\wedge Y 
}\]
commutes. Also Adams operations behave well with respect to the complexification map from $KO$-theory to $K$-theory, i.e.\ the diagram
\[\xymatrix@R-.3cm@C-.2cm@M+.1cm{ 
 X\ar[r]\ar[dr] & KO\wedge X \ar[d]_{c\wedge 1}\ar[r]^{\psi^k\wedge 1} & KO\wedge X \ar[d]^{c\wedge 1}\\
                & K\wedge X \ar[r]^{\psi^k\wedge 1}                    & K\wedge X 
}\]
commutes. Having $u\in K_*X$ we get $\psi^k(u)$ by
\[\xymatrix@R-.3cm@C-.2cm@M+.1cm{ 
S\ar[r]^(0.4){u} \ar@/_2pc/[rr]_{\psi^k(u)} & K\wedge X \ar[r]^{\psi^k \wedge X} & K\wedge X .
}\]
\subsubsection{The Kronecker pairing}
\begin{lemma}
We have $\langle \psi^k a, \psi^k b\rangle = \psi^k \langle a,b\rangle$ and $\psi^{k^{-1}}$ is adjoint to $\psi^k$ via the Kronecker pairing $\langle , \rangle: K^*X\otimes K_*X\rightarrow K_*$.
\end{lemma}
\begin{proof}
With $f$ representing $b\in K_*X$ and $g$ representing $a\in K^*X$ we have
\[\xymatrix@R-.3cm@C-.2cm@M+.1cm{ 
S \ar[dr]_{\psi^k(f)} \ar[r]^(.4){f} & K\wedge X \ar[d]^{\psi^k\wedge 1} \ar[r]^{1\wedge g} & 
K\wedge K\ar[d]^{\psi^k\wedge\psi^k}\ar[r]^(.6){\mu} & K\ar[d]^{\psi^k}\\
  & K\wedge X \ar[r]^{1\wedge \psi^k(g)} & K\wedge K \ar[r]^(.6){\mu} & K
}\]  
The line above gives $\langle a,b\rangle$, and the lower line $\langle \psi^k a,\psi^k b\rangle$, giving the first statement. For the second one we know that $\psi^k$ is a ring map, and with $b'=\psi^{k^{-1}} b$ we get
$$\langle \psi^k a, b\rangle=\langle \psi^k a, \psi^k b'\rangle = \psi^k \langle a,b'\rangle= \langle a,b'\rangle = \langle a,\psi^{k^{-1}} b\rangle.$$ 
\end{proof}

\subsection{Calculating operations on $K_*\mathbb{CP}^{\infty}$}
\subsubsection{Adams operations on $K_*\mathbb{CP}^{\infty}$}
As for any complex oriented theory, we have $K^*\mathbb{CP}^{\infty}\cong \pi_*K[\![x]\!] \cong \mathbb{Z}_p[\![x]\!]$ (the last isomorphism is due to the $K(1)$-local point of view). With $\beta_i$ being dual to $x^i$ the statement in homology is $K_*\mathbb{CP}^{\infty}\cong \mathbb{Z}_p\langle \beta_i\rangle$. Next we want to apply the pairing
$$ K^*X\otimes K_*X \rightarrow \pi_* K $$
and show its compatibility with Adams operations. Given $a\in K_*X$ (i.e.\ $S\rightarrow K\wedge X$) and $x\in K^*X$ (i.e.\ $X\rightarrow K$) we have the commuting diagram
\[\xymatrix@R-.3cm@C-.2cm@M+.1cm{ 
S \ar[r]\ar@{.>}[dr]_{\psi^k_*a} & K\wedge X \ar[d]^{\psi^k\wedge 1} & & \\
  & K\wedge X \ar[r]^{1\wedge x} & K\wedge K \ar[r]^(.6){\mu} & K.
}\]
Since $\psi^k$ is a ring homomorphism and $\psi^k(L)=L^k$ for line bundles $L$, we have for the generator $x=1-L$ (here we take the tautological line bundle over $\mathbb{CP}^{\infty}$
$$ \psi^k(x^n)= (1-(1-x)^k)^n.$$
The Kronecker pairing gives us $\langle \psi^k x,a\rangle = \langle x, \psi^{k^{-1}} a\rangle$, the duality was $\langle x^k, \beta_k\rangle =1$, hence basis expansion gives us
$$ \psi^{k^{-1}} \beta_n =\sum_{j\ge 0} \langle x^j, \psi^{k^{-1}}\beta_n\rangle \beta_j $$
with $\langle x^j,\psi^{k^{-1}} \beta_n\rangle = \langle \psi^k x^j,\beta_n\rangle =\langle (1-(1-x)^k)^j,\beta_n\rangle$.
We calculate 
\begin{eqnarray*}
 (1-(1-x)^3)^j &= & x^j(3+x(x-3))^j=x^j\sum_{s=0}^j \binom{j}{s} 3^{j-s} x^s(x-3)^s \\ 
               &= & x^j \sum_{s=0}^j \binom{j}{s}3^{j-s} x^s\sum_{t=0}^s \binom{s}{t}x^t (-3)^{s-t}
\end{eqnarray*}
Hence we have
$$ \langle \psi^3 x^j, \beta_i\rangle = (-1)^{i-j} \sum_{s+t=i-j} \binom{j}{s}\binom{s}{t} 3^{j-t} .$$
Explicitly we have
\begin{eqnarray*}
(3x-3x^2+x^3)^2  &=&   9x^2 - 18x^3 + 15x^4 - 6x^5 + x^6  \\
(3x-3x^2+x^3)^3  &=&  27x^3 - 81x^4 + 108x^5 - 81x^6 + 36x^7 - 9x^8 + x^9   \\
(3x-3x^2+x^3)^4  &=&  81x^4 - 324x^5 + 594x^6 - 648x^7 + 459x^8 - 216x^9 + 66x^{10}\\
& & -12x^{11} + x^{12}   \\
(3x-3x^2+x^3)^5  &=&   243x^5 - 1215x^6 + 2835x^7 - 4050x^8 + 3915x^9 - 2673x^{10} \\
& & +1305x^{11} - 450x^{12} + 105x^{13} - 15x^{14} + x^{15}  \\
(3x-3x^2+x^3)^6  &=&   729x^6 - 4374x^7 + 12393x^8 - 21870x^9 + 26730x^{10} - 23814x^{11}\\
& & +15849x^{12} - 7938x^{13} + 2970x^{14} - 810x^{15} + 153x^{16} - 18x^{17} +x^{18}  \\
(3x-3x^2+x^3)^7  &=&  2187x^7 - 15309x^8 + 51030x^9 - 107163x^{10} + 158193x^{11} \\ 
& &-173502x^{12} + 145719x^{13} - 95175x^{14} + 48573x^{15} - 19278x^{16}  \\
& & +5859x^{17}- 1323x^{18}+ 210x^{19} - 21x^{20} + x^{21}   \\
(3x-3x^2+x^3)^8  &=& 6561x^8 - 52488x^9 + 201204x^{10} - 489888x^{11} + 847098x^{12}\\
& & -1102248x^{13} + 1115856x^{14} - 896184x^{15} + 576963x^{16} - 298728x^{17}\\
& & +123984x^{18} - 40824x^{19} + 10458x^{20} - 2016x^{21}\\
& & + 276x^{22} - 24x^{23} +x^{24}\\
(3x-3x^2+x^3)^9  &=&  19683x^9 - 177147x^{10} + 767637x^{11} - 2125764x^{12}+ 4212162x^{13}\\
& & - 6337926x^{14} + 7501410x^{15} - 7138368x^{16} + 5535297x^{17}\\
& & - 3523257x^{18}+ 1845099x^{19} - 793152x^{20} + 277830x^{21} - 78246x^{22}\\
& & + 17334x^{23} -2916x^{24} + 351x^{25} - 27x^{26} + x^{27}   \\
(3x-3x^2+x^3)^{10}  &=&  59049x^{10} - 590490x^{11} + 2854035x^{12} - 8857350x^{13} + 19781415x^{14}\\
& & -33776028x^{15} + 45730170x^{16} - 50257260x^{17} + 45522405x^{18}\\
& & -34314030x^{19} + 21640365x^{20} - 11438010x^{21} + 5058045x^{22}\\
& & -1861380x^{23} + 564570x^{24} - 138996x^{25} + 27135x^{26}\\
& & - 4050x^{27} +435x^{28} - 30x^{29} + x^{30}  .
\end{eqnarray*}
Ordering the terms we get
\begin{eqnarray*}
\psi^{3^{-1}} \beta_1 &=& 3\beta_1\\
\psi^{3^{-1}} \beta_2 &=& -3\beta_1+9\beta_2\\
\psi^{3^{-1}} \beta_3 &=& \beta_1-18\beta_2+27\beta_3\\
\psi^{3^{-1}} \beta_4 &=& 15\beta_2-81\beta_3+81\beta_4\\
\psi^{3^{-1}} \beta_5 &=& -6\beta_2+108\beta_3-324\beta_4+243\beta_5\\
\psi^{3^{-1}} \beta_6 &=& \beta_2-81\beta_3+594\beta_4-1215\beta_5+729\beta_6\\
\psi^{3^{-1}} \beta_7 &=& 36\beta_3-648\beta_4+2835\beta_5-4374\beta_6+2187\beta_7\\
\psi^{3^{-1}} \beta_8 &=& -9\beta_3+459\beta_4-4050\beta_5+12393\beta_6-15309\beta_7+6561\beta_8\\
\psi^{3^{-1}} \beta_9 &=& \beta_3-216\beta_4+3915\beta_5-21870\beta_6+51030\beta_7-52488\beta_8+19683\beta_9\\
\psi^{3^{-1}} \beta_{10} &=& 66\beta_4-2673\beta_5+26730\beta_6-107163\beta_7+201204\beta_8-177147\beta_9+59049\beta_{10}.
\end{eqnarray*}

\subsubsection{Adams operations on $K_*(\mathbb{CP}^{\infty}\times\mathbb{CP}^{\infty})$}
Analogously we have 
$$K^*(\mathbb{CP}^{\infty}\times\mathbb{CP}^{\infty})\cong \pi_*K[\![x,y]\!]\cong\mathbb{Z}_p[\![x,y]\!]$$ 
and with $\beta_i\otimes \beta_j$ being dual to $x^iy^j$, we have
$$K_*(\mathbb{CP}^{\infty}\times\mathbb{CP}^{\infty})\cong \mathbb{Z}_p\langle \beta_i\otimes \beta_j\rangle.$$
Now the pairing above gives us
\begin{eqnarray*}
\langle x^iy^j,\psi^{k^{-1}} (\beta_m\otimes \beta_n)\rangle & = & \langle \psi^k x^i \psi^k y^j, \beta_m\otimes \beta_n\rangle \\
  & = & \langle (1-(1-x)^k)^i(1-(1-y)^k)^j, \beta_m\otimes \beta_n\rangle 
\end{eqnarray*}
and by basis expansion we have
$$ \psi^{k^{-1}} (\beta_i\otimes \beta_j) = \sum_{m,n} \langle x^my^n,\psi^{k^{-1}} (\beta_i\otimes \beta_j)\rangle \beta_m\otimes \beta_n.$$
With the map $\mathbb{CP}^{\infty}\times \mathbb{CP}^{\infty} \stackrel{f}{\rightarrow} BSU$ classifying the virtual $SU$ bundle $(1-L_1)(1-~L_2)$, the module generators $\beta_i\otimes \beta_j$ are mapped by $f_*$ to algebra generators in $K_*BSU$. By naturality of the Adams operations this allows us to calculate the Adams operations on $K_*BSU$.
We can calculate the pairing factor by factor:
\begin{eqnarray*}
\psi^{3^{-1}} \beta_i\otimes \beta_j &=& \sum_{m,n} \langle \psi^3 (x^my^n), \beta_i\otimes \beta_j\rangle \beta_m\otimes \beta_n\\
 & = & \sum_{m,n} \langle \psi^3 x^m,\beta_i\rangle \langle \psi^3 y^n,\beta_j\rangle \beta_m\otimes\beta_n\\
 & = & \sum_m \langle\psi^3x^m,\beta_i\rangle \beta_m \bigotimes \sum_n\langle \psi^3y^n,\beta_j\rangle \beta_n\\
 & = & (\psi^{3^{-1}}\beta_i)\otimes (\psi^{3^{-1}}\beta_j)
\end{eqnarray*}

The following table contains the mod 2 coefficients $a_k$ of $\psi^{3^{-1}}\beta_i=\sum a_k \beta_k$.
\begin{center}
\begin{tabular}{|r|rrrrrrrrrr|}
\hline
 & $\beta_1$ & $\beta_2$ & $\beta_3$ & $\beta_4$ & $\beta_5$ & $\beta_6$ & $\beta_7$ & $\beta_8$ & $\beta_9$ & $\beta_{10}$ \\
 \hline                        
$\psi^{3^{-1}} \beta_1$ = &     1& & & & & & & & &       \\ 
$\psi^{3^{-1}} \beta_2$ = &     1&1& & & & & & & &       \\ 
$\psi^{3^{-1}} \beta_3$ = &     1& &1& & & & & & &       \\ 
$\psi^{3^{-1}} \beta_4$ = &      &1&1&1& & & & & &       \\ 
$\psi^{3^{-1}} \beta_5$ = &      & & & &1& & & & &       \\
$\psi^{3^{-1}} \beta_6$ = &      &1&1& &1&1& & & &       \\
$\psi^{3^{-1}} \beta_7$ = &      & & & &1& &1& & &       \\
$\psi^{3^{-1}} \beta_8$ = &      & &1&1& &1&1&1& &       \\  
$\psi^{3^{-1}} \beta_9$ = &      & &1& &1& & & &1&       \\
$\psi^{3^{-1}} \beta_{10}$ = &   & & & &1& &1& &1&1      \\
\hline
\end{tabular}
\end{center}

\subsubsection{Integral Adams operations on $K_*BSU$ and 2-structures}
First we recall from \cite{Laures02} the structure of $K_*BSU$: Let $L$ be the tautological line bundle over $\mathbb{CP}^{\infty}$ and $\beta_i\in K_{2i}\mathbb{CP}^{\infty}$ be dual to $c_1(L)^i$. As above let $f$ be the map $\mathbb{CP}^{\infty}\times \mathbb{CP}^{\infty}\rightarrow BSU$ that classifies the bundle $(1-L_1)(1-L_2)$ with $L_i$ the tautological line bundle over the $i^{th}$ factor. Now choose for each natural number $k$ and $1\le i\le k-1$ integers $n_k^i$ such that
$$ \sum_{i=1}^{k-1} n_k^i \binom{k}{i} = gcd\{\binom{k}{1},..,\binom{k}{k-1}\}.$$
Defining elements $d_k=\sum_{i=1}^{k-1} n_k^i f_*(\beta_i\otimes \beta_{k-i})\in K_{2k}BSU$ we get
$$ K_*BSU\cong \mathbb{Z}_2[d_2,d_3,...]. $$
With the notation $a_{ij}=f_*(\beta_i\otimes \beta_j)\in \pi_{2(i+j)} K\wedge BSU$ and the fact that $$K^*(\mathbb{CP}^{\infty}\times\mathbb{CP}^{\infty})\cong \mathbb{Z}_2[\![x,y]\!]$$ 
with $\beta_i\otimes \beta_j\in K_*(\mathbb{CP}^{\infty}\times \mathbb{CP}^{\infty}) $ being dual to $x^iy^j$, we calculate the Adams operations by applying basis expansion:
$$ \psi^k a_{ij} = \sum_{m,n} \langle x^my^n,\psi^k a_{ij}\rangle a_{mn}.$$
Here again the Kronecker pairing gives
\begin{eqnarray*}
\langle x^my^n, \psi^k f_*(\beta_i\otimes \beta_j)\rangle &=& \langle \psi^{k^{-1}} f^*(x^my^n),\beta_i\otimes \beta_j\rangle\\
& = & \langle \psi^{k^{-1}} (x^my^n),\beta_i\otimes \beta_j\rangle\\
& = & \langle \psi^{k^{-1}} x^m,\beta_i\rangle \cdot \langle \psi^{k^{-1}}y^n,\beta_j\rangle .
\end{eqnarray*}
It is easily seen that
$$ gcd\{\binom{k}{1},...,\binom{k}{k-1}\}=\begin{cases} p \quad \text{ for } k=p^s\\ 1 \quad \text{ else. }\end{cases}$$
The generators $a_{ij}$ defined above satisfy certain relations which we want to describe in the following.
\begin{definition} The binomial coefficients associated to the formal group law $F$
$$ \binom{k}{i,j}_F \in \pi_{2(i+j-k)} E $$
are defined by the equation
$$ (x +_F y)^k = \sum_{i,j} \binom{k}{i,j}_F x^i y^j .$$
\end{definition}
\begin{example}[$K$-theory and the multiplicative formal group law]
For $E=K$ and $F=\hat{\mathbb{G}}_m$ with $\hat{\mathbb{G}}_m(x,y)=x+y-v^{-1} xy$ we have
$$ (x+_{\hat{\mathbb{G}}_m} y)^k = (x+y-v^{-1} xy)^k=\sum_{s=0}^k \sum_{t=0}^s \binom{k}{s} \binom{s}{t} (-v)^{s-k}x^{k-s+t} y^{k-t}$$
and hence
$$ \binom{k}{i,j}_{\hat{\mathbb{G}}_m} = \binom{k}{2k-i-j}\binom{2k-i-j}{k-j} (-v)^{k-i-j}.$$
\end{example}
\begin{lemma}[Laures]
The following relations hold for $i,j,k$:
\begin{eqnarray*}
a_{0,0}=1 & ; & a_{0i}=a_{i0}=0 \text{ for all } i\neq 0\\
    a_{ij}& = & a_{ji}\\
\sum_{l,s,t} \binom{l}{s,t}_{\hat{G}_m} a_{j-s,k-t}a_{il} & = & \sum_{l,s,t} \binom{l}{s,t}_{\hat{G}_m} a_{lk}a_{i-s,j-t}.
\end{eqnarray*}
\end{lemma}
For our calculations we choose the following non-vanishing coefficients $n_k^i$ for our basis elements $d_k$:
\begin{center}
\begin{tabular}{|r|rrrrrrrrrr|}
\hline
k $\backslash$ i & 1 & 2 & 3 & 4 & 5 & 6 & 7 & 8 & 9 & 10\\
\hline
2    & 1 &   &   &   &   &   &   &   &   &   \\
3    & 1 &   &   &   &   &   &   &   &   &   \\
4    & -1& 1 &   &   &   &   &   &   &   &   \\
5    & 1 &   &   &   &   &   &   &   &   &   \\
6    & 1 & 1 & -1&   &   &   &   &   &   &   \\
7    & 1 &   &   &   &   &   &   &   &   &   \\
8    & 9 &   &   & -1&   &   &   &   &   &   \\
9    & -9&   & 1 &   &   &   &   &   &   &   \\
10   & 1 & 11&   &   & -2&   &   &   &   &   \\
\hline
\end{tabular}
\end{center}
\begin{eqnarray*}
\psi^{3^{-1}} d_2 &=& \psi^{3^{-1}} f_*(\beta_1\otimes\beta_1)= f_*(\psi^{3^{-1}} (\beta_1\otimes \beta_1))\\
                  &=& f_*(3\beta_1\otimes 3\beta_1)=9 f_*(\beta_1\otimes \beta_1)=9 d_2\\
\psi^{3^{-1}} d_3 &=& f_*(\psi^{3^{-1}} \beta_1\otimes \beta_2)= f_*(3\beta_1\otimes (-3\beta_1+9\beta_2))\\
                  &=& 27d_3-9d_2\\
\psi^{3^{-1}} d_4 &=& \psi^{3^{-1}}(-f_*(\beta_1\otimes\beta_3)+f_*(\beta_2\otimes \beta_2))\\
                  &=& -f_*(3\beta_1\otimes (\beta_1-18\beta_2+27\beta_3))+f_*((-3\beta_1+9\beta_2)\otimes (-3\beta_1+9\beta_2))\\
                  &=& -9a_{11}+54a_{12}-81a_{13} + 9 a_{11}-54a_{12}+81a_{22}=81(-a_{13}+a_{22})=81d_4.
\end{eqnarray*}
To calculate the Adams operation on the higher $d_k$, one has to invest the 2-structure condition from above. An equivalent way to handle this is to see $K\wedge BSU_+$ as a complex oriented ring theory with complex orientation $x_{K\wedge BSU_*}=(1\wedge\eta)_* x_K$. The classifying map
$$ (\mathbb{CP}^{\infty}\times\mathbb{CP}^{\infty})_*\stackrel{f_+}{\rightarrow} BSU_+ \stackrel{\eta\wedge 1}{\rightarrow} K\wedge BSU_+$$
can be regarded as a power series
$$ f(x,y)=1+\sum_{i,j\ge 1} b_{ij} x^iy^j \in (K\wedge BSU)^0(\mathbb{CP}^{\infty}\times\mathbb{CP}^{\infty})$$
for some $b_{ij}\in K_{2(i+j)}BSU$. Indeed we have $b_{ij}=a_{ij}$ (compare \cite{Laures02})
\begin{eqnarray*}
b_{ij} &=& \sum_{k,l\ge 1} b_{ij} (1\wedge \eta)_* \langle \beta_i\otimes \beta_j, x^ky^l\rangle \\
       &=& \langle (1\wedge \eta)_*\beta_i \otimes (1\wedge \eta)_* \beta_j, 1+\sum_{k,l\ge 1} b_{kl} x^ky^l\rangle\\
       &=& \langle (1\wedge \eta)_* (\beta_i\otimes \beta_j), f^*(\eta\wedge 1)\rangle\\
       &=& (\mu f(1\wedge \eta))_* (\beta_i\otimes \beta_j)=a_{ij}.
\end{eqnarray*}
\begin{lemma}
For the power series $f(x,y)$ above, the following is straightforward to check:
\begin{eqnarray*}
f(x,0) &=& f(0,y)=1\\
f(x,y) &=& f(y,x) \\
f(x,y)f(x+_{\hat{G}_m} y) &=& f(x+_{\hat{G}_m} y,z)f(y,z).
\end{eqnarray*}
In the sense of \cite{AndoHopkinsStrickland} such an $f$ is called a 2-structure.
\end{lemma}
Comparison of the coefficients of $x^2yz$ gives the relation
$$ a_{21}+2a_{22}=a_{11}^2 +a_{31}.$$
With this we calculate
\begin{eqnarray*}
\psi^{3^{-1}} d_5 &=& \psi^{3^{-1}} f_*(\beta_1\otimes \beta_4)=f_*(3\beta_1\otimes (15\beta_2-81\beta_3+81\beta_4))\\
                  &=& 243d_5-243a_{13}+45a_{12}\\
                  &=& 243d_5-243a_{13}+45a_{12} +243(2a_{22}+a_{21}-a_{11}^2-a_{31})\\
                  &=& 243d_5+486(-a_{13}+a_{22})+288a_{12}-243a_{11}^2\\
                  &=& 243d_5+486d_4+288d_3-243d_2^2.
\end{eqnarray*}
To calculate further operations we need some additional relations:
\begin{center}
\begin{tabular}{|c|c|}
\hline
coefficient of & relation \\
\hline
$x^2yz$        & $2a_{22}-a_{31}+a_{21}+a_{11}^2$\\
$x^3yz$        & $2a_{14}+a_{11}a_{12}-a_{13}-a_{23}$\\
$x^2y^2z$      & $6a_{14}-6a_{13}+2a_{22}-a_{11}^2-a_{11}$\\
$x^3yz^2$      & $3a_{33}-2a_{23}+a_{11}a_{13}-a_{11}a_{22}-a_{12}^2$\\
$x^4yz$        & $5a_{15}-2a_{24}-3a_{14}+2a_{11}a_{13}+a_{12}^2$\\
\hline
\end{tabular}
\end{center}
With this we compute:
\begin{eqnarray*}
\psi^{3^{-1}} d_6 &=& \psi^{3^{-1}} (f_*(\beta_1\otimes\beta_5)+f_*(\beta_2\otimes\beta_4)-f_*(\beta_3\otimes\beta_3))\\
                  &=& f_*(3\beta_1\otimes (-6\beta_2+108\beta_3-324\beta_4+243\beta_5))\\
                  & & +f_*((-3\beta_1+9\beta_2)\otimes(15\beta_2-81\beta_3+81\beta_4))\\
                  & & -f_*((\beta_1-18\beta_2+27\beta_3)\otimes (\beta_1-18\beta_2+27\beta_3))\\
                  &=& 729d_6 -1215a_{14}+243a_{23}+513a_{13}-189a_{22}-27a_{12}-a_{11}\\
                  & & [+243(2a_{14}-a_{23}+a_{11}a_{12}-a_{13})]\\
                  &=& 729d_6-729d_5+270a_{13}-189a_{22}+243a_{11}a_{12}-27a_{12}-a_{11}\\
                  & & [-81(-a_{13}+2a_{22}+a_{12}+a_{11}^2)]\\
                  &=& 729d_6-729d_5-351d_4+243d_2d_3-108d_3-81d_2^2-d_2.
\end{eqnarray*}
Analogously we can go on 
\begin{eqnarray*}
\psi^{3^{-1}} d_7 &=& \psi^{3^{-1}} f_*(\beta_1\otimes \beta_6)=3^7 d_7 +3a_{12}-243a_{13}+1782 a_{14}-3645a_{15}
\end{eqnarray*}
and reduce the term above to polynomials in the $d_k$.

\subsection{Bott's formula and cannibalistic classes}
Due to Bott \cite{Bot69}, one can calculate Adams operations on the Thom space by calculating them on the base space, multiplying with the cannibalistic class $\theta_k(E)$ and applying the Thom isomorphism. Here the Thom space is constructed with respect to the bundle $E$. Let
$$ i_!:K(X)\rightarrow \tilde{K}(X^E)$$
denote the Thom isomorphism; then we have
$$ \psi^k (i_! x) = i_! \theta_k(E)\psi^k(x)$$
for $x\in K(X)$.

\subsubsection{Stable cannibalistic classes in $K$}
In order to calculate Adams operations on $K_*MSU$ using the formula above, one has to calculate cannibalistic classes. These classes are introduced by Bott in \cite{Bot69} and are defined for complex vector bundles over compact spaces $X$ and are characterized by the properties
\begin{itemize}
\item $\theta^k(L)=1+L^*+...+(L^*)^{k-1}$ for all line bundles $L$
\item $\theta^k(\xi+\xi')=\theta^k(\xi)\theta^k(\xi')$ for all complex bundles $\xi,\xi'$.
\end{itemize}
In particular, this implies that for the trivial bundle of rank $n$ (simply denoted by $n$) we have $\theta^k(\xi +n)=k^n\theta^k(\xi)$. To define {\em stable operations} $\hat{\theta}^k$ we assume $k$ to be an odd number and set 
$$ \hat{\theta}^k(\xi):=\frac{\theta^k(\xi)}{k^{dim_{\mathbb{C}}\xi}} \in K(X).$$
Next we calculate the cannibalistic classes for the universal $SU$-bundle
\[\xymatrix@R-.3cm@C-.2cm@M+.1cm{ 
(1-L_1)(1-L_2)\ar[d] & \\
\mathbb{CP}^{\infty}\times \mathbb{CP}^{\infty} \ar[r] & BSU.
}\]
Notice that it is indeed an $SU$ bundle because the first Chern class vanishes:
\begin{eqnarray*}
c(1+L_1L_2-L_1-L_2) &=& \frac{1\cdot c(L_1L_2)}{c(L_1)c(L_2)}=\frac{1+x_1+x_2}{(1+x_1)(1+x_2)}\\
 &=& (1+x_1+x_2)(1-x_1+x_1^2...)(1-x_2+x_2^2...)\\
 &=& 1+(x_1+x_2-x_1-x_2)+...
\end{eqnarray*}

Since $\theta^k(1)=k$ and $\theta^k(L)=1+L^*+...+(L^*)^{k-1}=\frac{1-(L^*)^k}{1-L^*}$ (formally) we have
\begin{eqnarray*} 
\theta^k((1-L_1)(1-L_2)) & = & \theta^k(1)\theta^k(L_1L_2)\theta^k(-L_1)\theta^k(-L_2)\\
 & = & k\frac{(1-(L_1^*L_2^*)^k)(1-L_1^*)(1-L_2^*)}{(1-L_1^*L_2^*)(1-(L_1^*)^k)(1-(L_2^*)^k)}.
\end{eqnarray*}

\begin{remark}
Choosing $x=1-L_1$ and $y=1-L_2$ as the generators of 
$$K_*(\mathbb{CP}^{\infty}\times\mathbb{CP}^{\infty}) \cong \mathbb{Z}_2[\![x,y]\!]$$ 
we can change to another orientation $x'=1-\frac{1}{1-x}=-\sum_{k\ge 1} x^k$. Hence we get $x'=1-L_1^*$ and $y'=1-L_2^*$, respectively.
We compute
$$ \hat{\theta}^k((1-L_1)(1-L_2))= k \frac{q_k(x'+y'-x'y')}{q_k(x')q_k(y')}$$
where $q_k(x')=\frac{1-(1-x')^k}{x'}$.
\end{remark}
We notice that $1-L=-L\otimes (1-L^*)$ and with the notations $x=1-L_1$ and $y=1-L_2$ we get
$$L_1^*=(1-x)^{-1} \text{ and } L_2^*=(1-y)^{-1}.$$ 
Hence for $k=3$ we can write
\begin{eqnarray*}
\theta^3((1-L_1)(1-L_2)) &=& 3 \frac{1+(1-x)^{-1}(1-y)^{-1}+(1-x)^{-2}(1-y)^{-2}}{(1+(1-x)^{-1}+(1-x)^{-2})(1+(1-y)^{-1}+(1-y)^{-2})}\\
& = & 3 \frac{(1-x)(1-y) +1+ (1-x)^{-1}(1-y)^{-1}}{((1-x)+1+(1-x)^{-1})((1-y)+1+(1-y)^{-1})}
\end{eqnarray*} 
and see that the cannibalistic class is invariant under $\psi^{-1}: L\mapsto L^*$.\\

\subsubsection{Calculating $\theta^3((1-L_1)(1-L_2))$}
\begin{lemma}
For the cannibalistic class $\theta^3$ we have the description
$$ \theta^3((1-L_1)(1-L_2))=3\frac{1+(1-x)(1-y)+(1-x)^2(1-y)^2}{(3-3x+x^2)(3-3y+y^2)} $$
and the coefficients of the power expansion
$$ \frac{1}{3-3x+x^2}=\sum_{k\ge 0} a_k x^k$$
satisfy the recurrence relation $a_0=a_1=\frac{1}{3}$ and $a_{n+2}=a_{n+1}-\frac{1}{3}a_n$, or more explicitly,
\begin{eqnarray*}
a_{6n}=a_{6n+1} &=& (-1)^n \,3^{-(3n+1)}\\
a_{6n+2}        &=& (-1)^n \,2\cdot 3^{-(3n+2)}\\
a_{6n+3}        &=& (-1)^n \,3^{-(3n+2)}\\
a_{6n+4}        &=& (-1)^n \,3^{-(3n+3)}\\
a_{6n+5}        &=& 0.
\end{eqnarray*}
\end{lemma}
\begin{proof}
Let $f(x)=\sum_{k\ge 0} a_k x^k$ denote the generating function of the recurrence relation $a_0=a_1=\frac{1}{3}$ and $a_{k+2}=a_{k+1}-\frac{1}{3}a_k$. Then we have
\begin{eqnarray*}
f(x) &=& \sum_{k\ge 0} a_k x^k = a_0+a_1x+\sum_{k\ge 2} (a_{k-1}-\frac{1}{3}a_{k-2})x^k\\
     &=& a_0+a_1x+ x (f(x) -a_0)-\frac{1}{3} x^2 f(x)\\
     &=& a_0 + f(x) (x-\frac{1}{3}x^2),
\end{eqnarray*}
hence 
$$ f(x)=\frac{1}{3-3x+x^2}.$$
\end{proof}
\begin{corollary}
The coefficients $c_{mn}$ of the power expansion of the cannibalistic class
$$  \theta^3((1-L_1)(1-L_2))=3 \frac{1+(1-x)(1-y)+(1-x)^2(1-y)^2}{(3-3x+x^2)(3-3y+y^2)} =\sum_{m,n\ge 0} c_{mn} x^my^n$$
are given by
\begin{eqnarray*}
 c_{mn}&=&9a_ma_n-9a_{m-1}a_n-9a_ma_{n-1}+3a_{m-2}a_n+15a_{m-1}a_{n-1}+3a_ma_{n-2}\\
       & & -6a_{m-2}a_{n-1}-6a_{m-1}a_{n-2}+3a_{m-2}a_{n-2}.
\end{eqnarray*}
The coefficents with negative indices are understood to be zero.
\end{corollary}
\begin{corollary}
The coefficients $c_{mn}$ are symmetric, i.e.\ $c_{mn}=c_{nm}$ and we have 
$$ c_{0n}=\begin{cases} 1 \text{ for } n=0 \\ 0 \text{ for } n\ge 1\end{cases} \quad \text{ and } \quad c_{1n}=3a_{n+1}. $$
If both indices are $\ge 2$ we have with $0\le i,k\le 5$:
$$c_{6m+i,6n+k}=(-1)^{m+n}\cdot 3^{-(3m+3n+ \left\lfloor \frac{i+k}{2} \right\rfloor)} b_{ik}$$
with
$$ b_{ik}=\begin{cases} 2 \text{ if } i-k=0 \\ 1 \text{ if } i-k=\pm 1,\pm 2\\ 0 \text{ if } i-k=\pm 3 \\ -1 \text{ if } i-k=\pm 4, \pm 5. \end{cases}$$
\end{corollary}

\begin{corollary}
We can also write this as
$$ c_{mn}= 3^{-\left\lfloor \frac{m+n}{2} \right\rfloor} \begin{cases} 2 \quad \text{ if } m-n\equiv 0           & \mod 6\\
                                                                     1   \quad \text{ if } m-n\equiv \pm 1,\pm 2 & \mod 12 \\
                                                                     0   \quad \text{ if } m-n\equiv 3           &  \mod 6\\
                                                                     -1  \; \text{ if } m-n\equiv \pm 4,\pm 5   & \mod 12 \end{cases}$$
for positive indices, whereas
$$ c_{0n}=\begin{cases} 1 \text{ for } n=0\\ 0 \text{ else. }\end{cases} $$                                                                    
\end{corollary}

\subsection{Spherical classes in $K_*MSU$}
We are now able to calculate spherical classes in $K_* MSU$. For this purpose we collect all relevant notions: Let $a=\tau^*b\in K^0MSU$ be an arbitrary class, $f$ the classifying map of the virtual $SU$-bundle $(1-L_1)(1-L_2)$ and 
$$d_k=\sum_{i=1}^{k-1} n_k^i f_*(\beta_i\otimes \beta_{k-i})$$
the generators of $K_*BSU\cong \pi_*K[d_2,d_3,...]$. Writing 
$$\theta^3((1-L_1)(1-L_2))=\sum_{m,n\ge 0} c_{mn} x^my^n$$
for the cannibalistic class of the $SU$-bundle above we get:
\begin{eqnarray*}
\langle a,\psi_M^{3^{-1}} d_k\rangle &=& \langle \psi^3_M(\tau^*b), d_k\rangle =\langle \theta^3\tau\psi_B^3(b), d_k \rangle\\
&=& \langle  \theta^3((1-L_1)(1-L_2)) \psi_B^3(b), \sum_{i=1}^{k-1} n_k^i f_*(\beta_i\otimes \beta_{k-i}) \rangle\\
&=& \sum_{i=1}^{k-1} n_k^i \sum_{m,n\ge 0} c_{mn}\langle f^*(\psi^3_B(b)), \beta_{i-m}\otimes \beta_{k-i-n}\rangle\\
&=& \sum_{i=1}^{k-1} n_k^i \sum_{m,n\ge 0} c_{mn} \langle b, \psi_B^{3^{-1}} f_*(\beta_{i-m}\otimes \beta_{k-i-n})\rangle.
\end{eqnarray*}

\begin{lemma}
The Adams operations on $K_*MSU$ are computable via the formula
$$ \psi_M^{3^{-1}} \Phi_* d_k = \Phi_*\big( \sum_{m,n\ge 0} c_{mn} \sum_{i=1}^{k-1} n_k^i \psi_B^{3^{-1}} f_*(\beta_{i-m}\otimes \beta_{k-i-n})\big),$$
where $\Phi_*$ is the Thom isomorphism.
\end{lemma}
Sample calculations (dropping the Thom isomorphism from the notation and writing $\psi^{3^{-1}}_M$ for the Adams operation on the level of the Thom spectrum) give with respect to the $n_k^i$ chosen above:
\begin{eqnarray*}
\psi_M^{3^{-1}} d_2 &=& 9d_2+\frac{2}{3}\\
\psi_M^{3^{-1}} d_3 &=& 27d_3-9d_2+\frac{1}{3}\\
\psi_M^{3^{-1}} d_4 &=& 81d_4+2 d_2 +\frac{1}{3}\\
\psi_M^{3^{-1}} d_5 &=& 243d_5+486d_4+288d_3-243d_2^2.
\end{eqnarray*}

\subsubsection{Construction of spherical classes}
Modulo 2 and omitting the Thom isomorphism, we get:
\begin{eqnarray*}
\psi_M^{3^{-1}} d_2 &=& d_2\\
\psi_M^{3^{-1}} d_3 &=& d_3+d_2+1\\
\psi_M^{3^{-1}} d_4 &=& d_4+1\\
\psi_M^{3^{-1}} d_5 &=& d_5+d_2^2.
\end{eqnarray*}
These calculations give the following spherical classes modulo 2:
\begin{eqnarray*}
z_4    &=& d_2\\
z_{12} &=& d_3^2+d_5+d_4+d_2^2\\
z_{16} &=& d_4^2+d_4\\
z_{20} &=& d_5^2+d_2^2d_5.
\end{eqnarray*}

\begin{remark}
We observe that there is no spherical class in degree 6 and conjecture that there is no spherical class in degree $4k+2$.
\end{remark}
\begin{corollary}
Since $\pi_6 T_{\zeta}=0$, this gives $MSU_6=0$ in the $K(1)$-local world.
\end{corollary}

\subsubsection{Lifting mod $p$ spherical classes}
Having mod $p$ spherical classes we are interested in getting integral spherical classes and constructing a spherical class basis for $K_*MSU$. We use the following algebraic lemma.
\begin{lemma}
Assume $A$ and $B$ to be $p$-complete. Let $f:A\rightarrow B\cong \mathbb{Z}_p[b_i]$ be such that there are $a_i\in A$ with $f(a_i)\equiv b_i$ modulo $p$. Then $f: A\rightarrow B$ is surjective.
\end{lemma}
We want to apply this lemma for $A=K_*MSU$ and assume that we have a basis of $\mod p$ spherical classes $A\cong \mathbb{Z}_p[a_i]$ with $\psi^3a_i\equiv a_i$ modulo $p$.
\begin{proposition}
There are elements $b_i$ such that
\begin{enumerate}
\item $A\cong \mathbb{Z}_p[b_i]$ and
\item $\psi^3 b_i=b_i$
\end{enumerate}
\end{proposition}
\begin{proof}
We make use of the bootstrap method: Assume $\psi^3a_i=a_i+pa'$ with $a'=\sum c_j a_j$. Then we have
$$ \psi^3 pa'= p\psi^3 a'= p\psi^3 (\sum c_i a_i)= p (\sum c_i \psi^3 a_i) + p^2 a'',$$
and go on the same way.
\end{proof}

\subsection{Umbral calculus}
\subsubsection{Mahler series in $p$-adic analysis}
The binomial polynomials define continuous functions
$$ \binom{\cdot}{k}: \mathbb{Z}_p\rightarrow \mathbb{Z}_p, \qquad x\mapsto \binom{x}{k}.$$
Since $\mathbb{N}$ is dense in $\mathbb{Z}_p$, we have $\left\| \binom{\cdot}{k} \right\| =\sup_{\mathbb{N}} \left| \binom{n}{k} \right| \le 1$. Because of $\binom{k}{k}=1$, equality holds in fact.
In $p$-adic analysis we know that for a given sequence $(a_i)_{i\ge 0}$ in $\mathbb{C}_p$ with $|a_i|\rightarrow 0$, the series $\sum_{k\ge 0} a_k \binom{\cdot}{k}$ is a continuous function $f:\mathbb{Z}_p\rightarrow \mathbb{C}_p$. It is quite remarkable that conversely, every continuous function $\mathbb{Z}_p\rightarrow \mathbb{C}_p$ can be represented this way. This result has been obtained by Mahler.
\begin{definition}
A Mahler series is a series $\sum_{k\ge 0} a_k \binom{\cdot}{k}$ with coefficients $|a_k|\rightarrow 0$ in $\mathbb{C}_p$.
\end{definition}
With the notation of the norm $\left\| f \right\|=\sup_{\mathbb{Z}_p} |f(x)|$ and the finite-difference operator $\nabla$
$$ (\nabla f)(x)=f(x+1)-f(x),$$
and its k-fold iterated version $\nabla^k$, we have:
\begin{theorem}[Mahler]
Let $f:\mathbb{Z}_p\rightarrow \mathbb{C}_p$ be a continuous function and put $a_k= \nabla^k f(0)$. Then $|a_k| \rightarrow 0$, and the series $\sum_{k\ge 0} a_k \binom{\cdot}{k}$ converges uniformly to $f$. Moreover $\left\| f \right\|=\sup_{k\ge 0} |a_k|$.
\end{theorem}
\subsubsection{The ring of numerical polynomials}
Let $A$ denote the ring 
$$A:= \{f\in\mathbb{Q}[\omega]\text{ such that } f(\mathbb{Z})\subset\mathbb{Z}\},$$
which we call the ring of numerical poynomials. 
\begin{remark}
This ring has been studied for a long time - historically Pascal considered elements $\binom{w}{i}=\frac{w(w-1)...(w-i+1)}{i!}$ and Fermat studied $\frac{w^p-w}{p}$ for $p$ a prime. In fact Newton found out that $1,w,\binom{w}{2},\binom{w}{3},...$ are a basis for $A$. 
\end{remark}
In $p$-adic analysis the $p$-completion of $A$, i.e.\
$$ \hat{A}_p=\{f\in \mathbb{Q}_p[\![\omega]\!]: f(\mathbb{Z}_p)\subset \mathbb{Z}_p\} $$
is, by a theorem of Mahler, the ring of continuous functions $f:\mathbb{Z}_p\rightarrow \mathbb{Z}_p$, and its elements can be written as Mahler series
$$ f(\omega)=\sum_{i=0}^{\infty} a_i \binom{\omega}{i} \quad\text{with}\quad a_i\rightarrow 0.$$
Integrally we can identify $K_0\mathbb{CP}^{\infty}\cong A$, i.e.\ the $K$-homology of $\mathbb{CP}^{\infty}$ equals the ring of numerical polynomials. The duality 
$$K^0\mathbb{CP}^{\infty}\cong \Hom(K_0\mathbb{CP}^{\infty},\mathbb{Z})$$
is given as follows: The series $\sum a_i t^i\in \mathbb{Z}[\![t]\!]\cong K^0\mathbb{CP}^{\infty}$ maps to the homomorphism given by $\binom{w}{i}\mapsto a_i$ on basis elements.

\subsubsection{Alternative description of the Adams operations}
We have seen that $K_*\mathbb{CP}^{\infty}$ is the ring of continuous functions on $\mathbb{Z}_p$ which are given as Mahler series. Its module generators $\beta_i\in K_{2i}\mathbb{CP}^{\infty}$ represent the function $\beta_i(T)=\binom{T}{i}$. Application of a base change leads to an interesting observation: At the prime 2 we have:
\begin{eqnarray*}
\binom{3T}{1} & = & 3\binom{T}{1}\\
\binom{3T}{2} & = & 9\binom{T}{2}+3\binom{T}{1}\\
\binom{3T}{3} & = & 27\binom{T}{3}+18\binom{T}{2}+\binom{T}{1}\\
\binom{3T}{4} & = & 81\binom{T}{4}+81\binom{T}{3}+15\binom{T}{2}\\
\binom{3T}{5} & = & 243\binom{T}{5}+324\binom{T}{4}+108\binom{T}{3}+6\binom{T}{2}\\
\binom{3T}{6} & = & 729\binom{T}{6}+1215\binom{T}{5}+594\binom{T}{4}+81\binom{T}{3}+\binom{T}{2}
\end{eqnarray*}
and this is exactly the Adams operation $\psi^{3^{-1}}$ on $\beta_k$ with respect to the generator $x=L-1\in K^*\mathbb{CP}^{\infty}$. The generator used before results in the same operation up to an alternating sign. 
\begin{lemma} The Adams operation $\psi^{3^{-1}}$ on $K_*\mathbb{CP}^{\infty}$ is given by
$$ \psi^{3^{-1}} \beta_i(T)=\beta_i(3T),$$
or, equivalently as the Mahler series
$$ \psi^{3^{-1}} \binom{T}{i}=\binom{3T}{i}=\sum_{j\ge 1} a_j \binom{T}{j},$$
where
$$ a_j=\sum_{s+t=i-j} \binom{j}{s}\binom{s}{t} 3^{j-t}.$$
\end{lemma}
\begin{proof}
Due to Mahler's theorem, the $j^{th}$ coefficient satisfies
$$ a_j=\nabla^j \binom{3x}{i}\big|_{x=0}, $$
i.e.\ it can be expressed using the $j$-fold iterated finite difference operator. Starting the calculations we get
$$ \nabla \binom{3T}{i}= 3\binom{3T}{i-1}+3\binom{3T}{i-2}+\binom{3T}{i-3}. $$
Comparing this to the calculation of the Adams operation in $K^*\mathbb{CP}^{\infty}$ with respect to the generator $x=L-1$ we get
$$ \psi^3 x= 3x+3x^2+x^3 $$
and see that taking the $j^{th}$ power of $3x+3x^2+x^3$ is exactly the same as taking the $j^{th}$ iterated finite difference operator $\nabla^j$. Hence the calculations coincide and the claim follows.
\end{proof}

\begin{corollary}
For a 2-adic unit $k\in\mathbb{Z}^{\times}_2$ we have 
$$\left(\psi^{k^{-1}} \beta_i\right)(T) =\beta_i(kT)=\binom{kT}{i}.$$
\end{corollary}
\begin{proof}
Since $3$ is a topological generator of $\mathbb{Z}_2^{\times}$, the sequence $a_n=3^n$ contains a subsequence $(a_{i_n})_n$ converging to $k$, and we have
\begin{eqnarray*}
\psi^{k^{-1}} \beta_i (T) &=& \lim_n \psi^{a_{i_n}^{-1}} \beta_i(T)=\lim_n \psi^{3^{-a_{i_n}}} \beta_i (T)\\
 &=& \lim_n \psi^{3^{-1}}\cdots \psi^{3^{-1}} \beta_i(T) = \lim_n \beta_i (3^{a_{i_n}} T) \\
 &=& \beta_i (kT).
\end{eqnarray*}
\end{proof}

\section{Open questions and concluding remarks}
Working towards a full $E_{\infty}$ splitting of the $K(1)$-local bordism spectrum $MSU$ one has to know that the quotient
$ M{\rm SU}/T_{\zeta}$
is free in the $K(1)$-local stable homotopy category. This implies a basis corresponding to the spherical classes which is geometrically realized as free $E_{\infty}$ summands $TS^0$. I conjecture that $MSU$ splits at $p=2$ as
$ M{\rm SU} \cong T_{\zeta} \wedge \bigwedge_{i=1}^{\infty} TS^0 ,$
which looks very similar to the $E_{\infty}$ splitting of $M{\rm Spin}$. Maybe the $E_{\infty}$ comparison map
$ M{\rm SU} \rightarrow M{\rm Spin}$
is close to an $E_{\infty}$ equivalence. An indication for this is the vanishing of $MSU_6$ after $K(1)$-localization. Integrally this is false, because the comparison map is neither injective nor surjective on the level of homotopy. To give examples we mention that $\mathbb{HP}^2$ is not a complex manifold, but it is spin, thus the comparison map is not surjective. 
While the 3-connected manifold $\mathbb{HP}^2$ represents a non-trivial spin bordism class, it does not admit a stably complex structure (since its signature is odd, cf.\ \cite{CF66a}). On the other hand we have $M{\rm Spin}_6=0$, but by construction the projective variety 
$$ \mathcal{K4}=\{ z\in\mathbb{CP}^4 | z_0^5+z_1^5+z_2^5+z_3^5+z_4^5 = 0\} $$
has vanishing first Chern class and represents a non-zero class in $M{\rm SU}_6$. Recall from \cite{ABP66} that an $SU$-manifold is null bordant if and only if its Chern numbers and $KO$-characteristic numbers vanish.\\

\begin{center}
\begin{tabular}{|r||r|r|r|r|r|r|r|r|r|}
\hline
$n$        &                   0 &                1 &               2 &                3 &  4  &     5  &            6   &            7  &       8 \\
\hline 
$\pi_n^s$&  $\mathbb{Z}_{(2)}$ & $\mathbb{Z}/2$  & $\mathbb{Z}/2$  & $\mathbb{Z}/8$  & $0$ & $0$  & $\mathbb{Z}/2$ &$\mathbb{Z}/16$&$\mathbb{Z}/2\oplus\mathbb{Z}/2$ \\
$KO_n$ &  $\mathbb{Z}_{(2)}$ & $\mathbb{Z}/2$  & $\mathbb{Z}/2$  & $0$ & $\mathbb{Z}_{(2)}$ & $0$ & $0$ & $0$ &  $\mathbb{Z}_{(2)}$ \\
$M{\rm SU}_n$  & $\mathbb{Z}_{(2)}$ & $\mathbb{Z}/2$  & $\mathbb{Z}/2$  & $0$ & $\mathbb{Z}_{(2)}$ & $0$ & $\mathbb{Z}_{(2)}$ &$0$ & $\mathbb{Z}_{(2)}\oplus \mathbb{Z}_{(2)}$\\
$M{\rm Spin}_n$ & $\mathbb{Z}_{(2)}$ & $\mathbb{Z}/2$  & $\mathbb{Z}/2$  & $0$ & $\mathbb{Z}_{(2)}$ & $0$ & $0$ &$0$ & $\mathbb{Z}_{(2)}\oplus \mathbb{Z}_{(2)}$\\
\hline
\end{tabular}\\[3mm]
The $2$-primary part of some relevant homotopy groups\\[3mm]
\end{center}
Desiring progress in a full $SU$-splitting, one has to get a better understanding of the spherical classes. One approach is to apply better arithmetic techniques, another approach is to interpret the Adams operations as a precomposition of automorphisms as in the example of $K_*\mathbb{CP}^{\infty}$. Again another access is the study of symmetric $2$-cocycles in the sense of \cite{Laures02} and \cite{AndoHopkinsStrickland}.
This is an interesting arithmetic problem and it poses quite a challenge to calculate the corresponding spherical classes.\\[3mm]
Another problem which is not solved yet is a $K(1)$-local additive decomposition of $M\rm{SU}$ in terms of $K$-theory. In \cite{Pengelley} Pengelley gives a 2-local additive splitting of $MSU$
$$ MSU_{(2)}\cong \bigvee_i \Sigma^{d_i} BoP \vee \bigvee_j \Sigma^{d_j'} BP $$
into a wedge of suspensions of Brown-Peterson spectra $BP$ and a wedge of suspensions of other indecomposable spectra $BoP$, which bear similarities to the $BP$ spectrum and connective $K$-theory $ko$. In \cite{Hovey97} Marc Hovey conjectures that the spectrum $BoP$ splits into a wedge of suspensions of $K$-theory spectra $K$ and $KO$, so that in this case $MSU$ splits additively like
$$ MSU \cong \bigvee K \vee \bigvee KO.$$
Using results from \cite{Hopkins98}, more specifically that
$ \pi_*K\wedge K \cong \Hom_{cts}(\mathbb{Z}_2^{\times},\pi_*K)$
and
$ \pi_*K\wedge KO \cong \Hom_{cts}(\mathbb{Z}_2^{\times}/\{\pm 1\}, \pi_* K),$
we see that
\begin{eqnarray*}
K_*MSU & \cong & \bigoplus K_*K \oplus \bigoplus K_*KO\\
       & \cong & \bigoplus \Hom_{cts}(\mathbb{Z}_2^{\times},\pi_*K) \oplus \bigoplus \Hom_{cts}(\mathbb{Z}_2^{\times}/\{\pm 1\}, \pi_* K).
\end{eqnarray*}
It is highly desirable to get a precise additive splitting. Such a description would offer comforting methods to calculate Adams operations.

%
%

\vspace*{10mm}
\parbox{\linewidth}{
\textsc{Holger Reeker}\newline
\textsc{Ruhr-Universit\"at Bochum}\newline
\textsc{Universit\"atsstra{\ss}e 150}\newline
\textsc{44801 Bochum}\newline
\textsc{Germany}\newline
\phantom{ }\newline
\textsc{Email: }\texttt{holger.reeker@rub.de}\hfill}

\begin{thebibliography}{HodSna75}

\bibitem[Ada74]{Adams74}
J. F. Adams, Stable Homotopy and Generalized Homology, Chicago Lectures in Mathematics, 1974.

\bibitem[Ada78]{Adams78}
J. F. Adams, Infinite Loop Spaces, Annals of Mathematics Studies 90, Princeton University Press, 1978.

\bibitem[AHS01]{AndoHopkinsStrickland}
M. Ando, M.J. Hopkins, N.P. Strickland, Elliptic spectra, the Witten genus and the theorem of the cube, Invent. math. 146, pp. 595-687 (2001).

\bibitem[ABP66]{ABP66}
D.W.\ Anderson, E.H.\ Brown and F.P.\ Peterson, ${\rm SU}$-cobordism, ${\rm KO}$-characteristic numbers, and the Kervaire invariant, Ann. of Math. (2) 83 1966 54--67. 


\bibitem[Ati66]{Atiyah:poweroperations}
M.F. Atiyah, Power operations in {$K$}-theory,
Quart. J. Math. Oxford Ser. (2) 17, p. 165-193, 1966.

\bibitem[BMMS86]{BMMS86}
R.R.\ Bruner, J.P. May, J.E. McClure and M. Steinberger, {\em $H_{\infty}$ Ring Spectra and their Applications}, LNM 1176, Springer-Verlag, 1986.

\bibitem[Bot69]{Bot69}
R.\ Bott, Lectures on $K(X)$, 1969.


\bibitem[Bou79]{Bousfield79}
A.K.\ Bousfield, The localization of spectra with respect to homology,
Topology 18 (1979), p. 257-281.

\bibitem[Bou96]{Bousfield}
A.K.\ Bousfield, On $\lambda$-Rings and the $K$-theory of Infinite Loop Spaces.


\bibitem[CF66a]{CF66a}
P.E. Conner, E.E. Floyd, The Relation of Cobordism to $K$-Theories, LNM 28, 1966.

\bibitem[CF66b]{CF66b}
P.E. Conner, E.E. Floyd, Torsion in SU-Bordism, Memoirs of the AMS Number 60, 1966. 


\bibitem[EKMM97]{EKMM}
A.D. Elmendorf, I. Kriz, M.A. Mandell and J.P. May,
Rings, Modules and Algebras in Stable Homotopy Theory. With an appendix by M. Cole, Math. Surveys Monogr. vol. 47,
AMS 1997.

\bibitem[Ful98]{Fulton}
William Fulton, Intersection theory (Second edition), Ergebnisse der Mathematik und ihrer Grenzgebiete, Springer-Verlag, Berlin, 1998.

\bibitem[GHMR]{GHMR}
Paul Goerss, Hans-Werner Henn, Mark Mahowald and Charles Rezk, A resolution of the K(2)-local sphere at the prime 3, Annals of Mathematics, Vol. 162, 2005. 



\bibitem[HBJ92]{Hirzebruch:Berger:Jung}
Friedrich Hirzebruch, Thomas Berger and Rainer Jung,
Manifolds and modular forms, 1992.

\bibitem[HG94]{HG94}
Michael J. Hopkins and B.H. Gross, The rigid analytic period mapping, Lubin-Tate space, and stable homotopy theory, Bulletin of the American Mathematical society, Vol. 30, Number 1, January 1994, pp. 76-86.

\bibitem[Hi56]{Hi56}
Friedrich Hirzebruch, Neue topologische Methoden in der algebraischen Geometrie, Ergebnisse der Mathematik und ihrer Grenzgebiete (N.F.), Heft 9. Springer-Verlag, Berlin-G\"ottingen-Heidelberg, 1956.

\bibitem[HM07]{HM07} Rebekah Hahn and Stephen Mitchell, Iwasawa theory for $K(1)$-local spectra, Trans. Amer. Math. Soc. 359 (2007), No. 11, pp. 5207-5238.

\bibitem[Hop94]{Hopkins94}
Michael J. Hopkins, Topological Modular Forms, the Witten Genus and the Theorem of the Cube, Proceedings of the International Congress of Mathematics, 1994.

\bibitem[Hop96]{Hopkins96}
Michael J. Hopkins,
Course notes for elliptic cohomology, 1996. 

\bibitem[Hop98]{Hopkins98}
Michael J. Hopkins,
$K(1)$-local $E_{\infty}$ ring spectra, 1998.

\bibitem[Ho97]{Hovey97}
Mark Hovey, $v_n$-elements in ring spectra and applications to bordism theory, Duke Mathematical Journal, Vol.\ 88 No.\ 2, 1997.

\bibitem[HS75]{HodgkinSnaith}
Luke Hodgkin and Victor Percy Snaith, Topics in K-theory: I. The equivariant K\"unneth theorem in K-theory, II. Dyer-Lashof operations in K-theory, LNM 496, 1975.

\bibitem[HS98]{HS}
Michael J. Hopkins and Jeffrey H. Smith, Nilpotence and Stable Homotopy Theory II, Annales of Mathematics, Second Series, Vol. 148 (1998), No. 1, pp. 1-49.

\bibitem[HSS00]{HSS00}
Mark Hovey, Brooke Shipley and Jeff Smith, Symmetric spectra, J. Amer. Math. Soc. 13 (2000), no. 1, 149--208.

\bibitem[Ko82a]{Kochman82a}
Stanley O. Kochman, Change of Bases in $H_*BU$, 
Amer. J. of Math. 106 (1984), pp. 233-254.

\bibitem[Ko82b]{Kochman82b}
Stanley O. Kochman, Polynomial generators for $H_*(BSU)$ and $H_*(BSO;\mathbb{Z}_2)$,
Proc. Amer. Math. Soc. 84 (1982), 149-154.

\bibitem[Ko93]{Kochman93}
Stanley O. Kochman, The Ring Structure of $BoP_*$,
Contemp. Math. 146, 1993.

\bibitem[KT06]{KoTa}
Akira Kono and Dai Tamaki,
Generalized Cohomology, Transl. of Mathematical Monographs Vol. 230, AMS (2006).

\bibitem[Lau01]{Habi}
Gerd Laures: An $E_\infty$-splitting of spin bordism with applications to real $K$-theory and topological modular forms, Habilitationsschrift Mai 2001, Universit\"at Heidelberg, 106 pages.  

\bibitem[Lau02]{Laures02} 
Gerd Laures,
Characteristic numbers from 2-stuctures on formal groups, 
J. Pure Appl. Algebra 172 (2002), pp. 239-256.

\bibitem[Lau03]{Laures03}
Gerd Laures, An $E_{\infty}$ splitting of spin bordism,
American Journal of Mathematics 125 (2003), pp. 977-1027.

\bibitem[Lau04]{Laures04}
Gerd Laures, $K(1)$-local topological modular forms, Inventiones mathematicae 157, pp. 371-403 (2004).

\bibitem[Li64]{Liulevicius}
Arunas Liulevicius, Notes on Homotopy of Thom Spectra,
Amer. J. Math. 86 (1964), p. 1-16.

\bibitem[MMSS01]{MMSS}
M.A. Mandell, J.P. May, S. Schwede and B. Shipley, Model Categories of Diagram Spectra, Proc. London Math. Soc. (3) 82 (2001), No. 2, pp. 441--512.

\bibitem[Mi94]{Miller}
Haynes Miller, Notes on Cobordism, MIT lecture notes typed by Dan Christensen and Gerd Laures, 1994.

\bibitem[Mit93]{Mitchell}
Stephen A. Mitchell, On p-adic topological K-theory, Algebraic K-theory and Algebraic Topology
(P. G. Goerss and J. F. Jardine, eds.), Kluwer, Dordrecht, 1993, pp. 107-204.


\bibitem[MR81]{MR81}
Mark Mahowald and Nigel Ray, A note on the Thom isomorphism, {\em Proc. Amer. Math. Soc. 82 (1981)}, pp. 307-308.

\bibitem[MS04]{McCSm}
James E. McClure and Jeffrey H. Smith,
Operads and cosimplicial objects: an introduction. Axiomatic, enriched and motivic homotopy theory, pp.\ 133--171, NATO Sci. Ser. II Math. Phys. Chem., 131, Kluwer Acad. Publ., Dordrecht, 2004.

\bibitem[Pen82]{Pengelley}
David J. Pengelley, The Homotopy Type of $M{\rm SU}$,
American Journal of Mathematics {\bf 104} (1982) no. 5, pp. 1101-1123.

\bibitem[Rav84]{Rav84}
Douglas C. Ravenel, Localization with Respect to Certain Periodic Homology Theories,
American Journal of Mathematics 106, pp. 351-414.

\bibitem[Re97]{Re97}
Charles Rezk, Notes on the Hopkins-Miller Theorem, Homotopy Theory via Algebraic Geometry and Group Representations (Proceedings of the 1997 Northwestern conference).

\bibitem[Re06]{Rezk}
Charles Rezk, Lectures on power operations, MIT lecture notes June 2006.

\bibitem[Rob00]{Robert}
Alain M. Robert, A Course in p-adic Analysis, Springer GTM 198, 2000.

\bibitem[Ru98]{Rudyak}
Yuli B. Rudyak, On Thom Spectra, Orientability, and Cobordism, Springer Monographs in Mathematics, 1998.


\bibitem[Sch08]{Schwede08}
Stefan Schwede, On the homotopy groups of symmetric spectra, Geom. Topol. 12 (2008), no. 3, 1313--1344.

\bibitem[Swi75]{Switzer}
Robert M. Switzer, Algebraic Topology - Homology and Homotopy, Springer Reprint of 1975.

\bibitem[Wa97]{Wa97}
Lawrence C. Washington, Introduction to Cyclotomic Fields, Second Edition, Springer GTM 83, 1997.

\end{thebibliography}
\end{document}